\documentclass[12pt]{amsart}

\usepackage{amsmath,amsthm,amssymb,enumerate,amsfonts,mathrsfs,amsxtra,amscd,latexsym,graphicx}
\usepackage[OT2,T1]{fontenc}
\usepackage{amsmath}
\usepackage{amsfonts}
\usepackage{amssymb}
\usepackage{amsthm}
\usepackage{hyperref}
\usepackage{enumerate}
\usepackage{cleveref}
\usepackage{mathrsfs}
\usepackage{lscape}
\usepackage{geometry}
\usepackage{multicol}
\usepackage{float}

\DeclareSymbolFont{cyrletters}{OT2}{wncyr}{m}{n}\DeclareMathSymbol{\Sha}{\mathalpha}{cyrletters}{"58}

\newtheorem{theorem}{Theorem}[section]
\newtheorem{lemma}[theorem]{Lemma}
\newtheorem{proposition}[theorem]{Proposition}

\newtheorem{definition}[theorem]{Definition}
\newtheorem{conjecture}[theorem]{Conjecture}
\Crefname{conjecture}{Conjecture}{Conjectures}
\theoremstyle{remark}
\newtheorem*{remark}{Remark}

\theoremstyle{plain}

\theoremstyle{plain}

\newcommand{\N}{\mathbb{N}}
\newcommand{\Div}{{\text {\rm div}}}
\newcommand{\Z}{\mathbb{Z}}
\newcommand\smod[1]{\ \left(\operatorname{mod} #1\right)}
\newcommand{\Q}{\mathbb{Q}}
\newcommand{\R}{\mathbb{R}}
\newcommand{\C}{\mathbb{C}}

\newcommand{\Jac}{\text {\rm Jac}}
\newcommand{{\D}}{\delta}
\newcommand{{\Conj}}{\text {\rm Conj}}

\newcommand{\Stab}{\operatorname{Stab}}
\newcommand{\eps}{\varepsilon}
\newcommand{\GL}{\operatorname{GL}}

\newcommand{\SL}{\operatorname{SL}}

\newcommand{\calH}{\mathscr{H}}

\newcommand{\SLZ}{\SL_2(\Z)}
\newcommand{\abcd}{\left(\begin{smallmatrix} a & b \\ c & d \end{smallmatrix}\right)}

\newcommand{\tr}{\operatorname{tr}}

\newcommand{\calQ}{\mathcal{Q}}

\newcommand{\rk}{\operatorname{rk}}

\newcommand{\Gal}{\operatorname{Gal}}

\newcommand{\calT}{\mathscr{T}}
\newcommand{\mult}{\operatorname{mult}}
\newcommand{\HH}{\mathfrak{H}}

\newcommand{\pr}{\operatorname{pr}}
\renewcommand{\t}{\mathbf{t}}
\newcommand{\ON}{\textsl{O'N}}
\newcommand{\Tr}{\operatorname{Tr}}

\newcommand{\vol}{\operatorname{vol}}
\newcommand{\Sel}{\operatorname{Sel}}
\newcommand{\Reg}{\operatorname{Reg}}
\newcommand{\ord}{\operatorname{ord}}
\newcommand{\Diff}{\operatorname{Diff}}
\newcommand{\g}{\mathscr{G}}

\numberwithin{equation}{section}
\numberwithin{table}{section}

\author{John F. R. Duncan, Michael H. Mertens, and Ken Ono}
\address{Department of Mathematics and Computer Science, Emory University, 400 Dowman Drive, Atlanta, GA 30322}
\email{john.duncan@emory.edu }
\address{Mathematisches Institut der Universt\"at zu K\"oln, Weyertal 86-90, D-50931 K\"oln} 
\email{mmertens@math.uni-koeln.de}
\address{Department of Mathematics and Computer Science, Emory University, 400 Dowman Drive, Atlanta, GA 30322}
\email{ono@mathcs.emory.edu }

\subjclass[2010]{11F22, 11F37}

\title{O'Nan Moonshine and Arithmetic}

\begin{document}

\begin{abstract} Answering a question posed by Conway and Norton in their seminal 1979 paper on moonshine,
we prove the existence of a graded infinite-dimensional module for the sporadic simple group of O'Nan, 
for which the McKay--Thompson series are weight $3/2$ modular forms.
The coefficients of these series may be expressed in terms of class numbers, traces of singular moduli, and central critical values of quadratic twists of weight 2 modular $L$-functions.
As a consequence, for primes 
$p$ dividing the order of the O'Nan group
we obtain congruences between 
O'Nan group character values and class numbers, $p$-parts of Selmer groups, and Tate--Shafarevich groups of
certain elliptic curves. This work represents the first example of moonshine involving arithmetic invariants of this type. 
\end{abstract}

\maketitle

\section{Introduction and Statement of Results}

The {sporadic} simple groups are the twenty-six exceptions to the classification \cite{MR2072045} of finite simple groups: those examples that aren't included in any of the natural infinite families.
It is natural to wonder where they appear, outside of the classification itself.

At least for the {\em monster}, being the largest of the sporadics, this question has an interesting answer. 
By the last decade of the last century, Ogg's observation \cite{Ogg_AutCrbMdl} on primes dividing the order of the monster, McKay's famous formula $$196884=1+196883,$$ and the much broader family of coincidences 
observed by Thompson \cite{Tho_FinGpsModFns,Tho_NmrlgyMonsEllModFn} and Conway--Norton \cite{CN79}, were proven by Borcherds \cite{borcherds_monstrous} to reflect the existence of a certain distinguished algebraic structure. This {\em moonshine module}, constructed by Frenkel--Lepowsky--Meurman \cite{FLMPNAS,FLMBerk,FLM}, admits a vertex operator algebra structure, has the monster as its full symmetry group, and has modular functions for traces. It is a cornerstone of {\em monstrous moonshine}, and indicates a pathway by which ideas from theoretical physics, and string theory in particular, may ultimately reveal a natural origin for the monster group and its curious connection to modularity. 

In addition to the monster itself, nineteen of the sporadic simple groups appear as quotients of subgroups of the monster. As such, we may expect that monstrous moonshine extends to them in some form. This is consequent upon the {\em generalized moonshine conjecture}, which was formulated by Norton \cite{generalized_moonshine} following preliminary observations of Conway--Norton \cite{CN79} and Queen \cite{MR628715}, and which has been recently proven in powerful work by Carnahan \cite{CarGM}. 

Certain more general analogues of monstrous moonshine have appeared in this century. In 2010, Eguchi--Ooguri--Tachikawa \cite{EOT} sparked a resurgence of interest in moonshine with their observation that the elliptic genus of a K3 surface---a trace function arising from a non-linear sigma model with K3 target---is, essentially, the sum of an indefinite theta function and a $q$-series whose coefficients are dimensions of modules for Mathieu's largest sporadic group, $M_{24}$. In fact,  this $q$-series is a mock modular form which, together with most of 
Ramanujan's mock theta functions, belongs to a family of distinguished examples \cite{omjt} arising from a family of finite groups. This is {\em umbral moonshine} \cite{UM,MUM,mumcor}, and the existence of corresponding umbral moonshine modules has been verified by Gannon \cite{Gannon} in the case of $M_{24}$, and in general by Griffin and two of the authors of this work \cite{DGO}. However, it must be noted that this theory is not yet on the same footing as monstrous moonshine, as suitable umbral counterparts to the moonshine module vertex operator algebra of Frenkel--Lepowsky--Meurman are not yet known in general. 

We refer the reader to \cite{FLM,MR2201600} for fuller discussions of monstrous moonshine, and to Gannon's book \cite{GannonMBM} for a broad perspective on the theory. The more recent review \cite{DGOSurvey} includes some umbral developments. We refer to \cite{PPV16,PPV17} for new work on the string theoretic interpretation of monstrous moonshine, and refer to \cite{ACH,CD17,DH,DO} for vertex algebraic constructions of some of the umbral moonshine modules.

Very recently, yet another form of moonshine has appeared in work of Harvey--Rayhaun \cite{HarveyRayhaun} which manifests a kind of half-integral weight counterpart to 
generalized moonshine for Thompson's sporadic group. This is known as {\em Thompson moonshine}. The existence of a corresponding module has been confirmed by Griffin and one of the authors \cite{GM16} (but in this case too, a vertex algebraic realization is yet to be found).

All the umbral groups are involved in the monster in some way, so we are left to wonder if there are counterparts to monstrous moonshine for the remaining six {\em pariah} sporadic groups: the \emph{Janko groups} $J_1, J_3,$ and $J_4$, the \emph{Lyons group} $Ly$, the {\em Rudvalis group} $Ru$, and the \emph{O'Nan group} $\ON$.  
Can moonshine shed light on these groups too? Conway and Norton asked this question 
(cf. p. 321 of \cite{CN79}) in their seminal 1979 paper:

\smallskip
\noindent
{\it ``Finally, we ask whether the sporadic simple groups that may not be involved in [the monster]... have moonshine properties.''}
\smallskip

\noindent
{This question is also Problem \#9 in the 1998 paper by Borcherds entitled ``Problems in Moonshine'' \cite{Borcherds98}.}

Rudvalis group analogues of the moonshine module were constructed in \cite{Duncan1,Duncan2}, but the physical significance of these structures is yet to be illuminated. In this work we present a new form of moonshine which reveals a role for the O'Nan group in arithmetic: as an organizing object for congruences between class numbers, $p$-parts of Selmer groups and Tate--Shafarevich groups of elliptic curves. (See e.g. \cite{Silverman1} or \cite{Tate} for background on elliptic curve arithmetic.) This is the first occurrence of moonshine of this type. Since $J_1$ is a subgroup of $\ON$ it suggests that at least two pariah groups play an active part in some of the deepest open questions in arithmetic. 

\subsection{Moonshine and Divisors}

Before describing our results in more detail we offer a conceptual number theoretic perspective which ties together some of the recent developments mentioned above. Suppose that $G$ is one of the finite groups appearing in the aforementioned cases of moonshine. 
Then we have an infinite-dimensional graded $G$-module, say $V^{G}$, which manifests a collection of modular forms,
one for each conjugacy class. For monstrous, umbral, and Thompson moonshine we have 
$$
V^G=\bigoplus_m V^G_m\ \  \overset{\rm moonshine}{\xrightarrow{\hspace{1in}}} \ \ ( f_{[g]})\in 
\begin{cases}  \ \bigoplus\limits_{[g]\in \Conj(G)} M_0^{!}(\Gamma_{[g]}) \ \ \ \ \ &{\text {\rm monstrous}}\\
\ \ \\
\bigoplus\limits_{[g]\in \Conj(G)} H_{\frac{1}{2}}(\Gamma_{[g]}) \ \ \ \ \ &{\text {\rm umbral, Thompson}}.
\end{cases}
$$
The defining feature of the $f_{[g]}$ is that their $m^{\rm th}$ coefficients equal the
graded traces $\tr(g | V^{G}_m)$.

In monstrous moonshine, the $f_{[g]}$ are Hauptmoduln for genus 0 groups $\Gamma_{[g]}$
(essentially level $o(g)$ congruence subgroups).
At the cusp $\infty$, they have Fourier expansion 
$$
f_{[g]}=q^{-1} + O(q)
$$
(note $q:=e^{2\pi i \tau}$ throughout), and are holomorphic at other cusps. In particular, this means that
 $\Div(f_{[g]})=cz-\infty$ for some $z\in X(\Gamma_{[g]})$ and some integer $c$. 
In contrast, the  $f_{[g]}$ in umbral and Thompson moonshine are not functions on modular curves, so it does not generally make sense to consider their divisors.
Instead, they are weight 1/2 harmonic Maass forms (with multiplier) for $\Gamma_{[g]}$, which means
that the McKay--Thompson series are generally {\it mock modular forms}, the holomorphic parts of the $f_{[g]}$.
Although they are not functions on these modular curves, it turns out that they actually encode even more information about divisors on
$X(\Gamma_{[g]})$.
For each discriminant $D$,
there is a map $\Psi_D$ for which
$$
V^{G} =\bigoplus_m V^G_m\ \ \overset{\rm moonshine}{\xrightarrow{\hspace{1in}}}\ \  (f_{[g]})
\ \  \overset{\Psi_D}{\xrightarrow{\hspace{.5in}}}
\ \ \ (\Psi_D(f_{[g]})) \in \bigoplus\limits_{[g]\in \Conj(G)} \mathcal{K}(\Gamma_{[g]}),
$$
where $\mathcal{K}(\Gamma_{[g]})$ is the field of modular functions for $\Gamma_{[g]}$. The $\Psi_D(f_{[g]})$ are {\it generalized Borcherds products} as defined by Bruinier and one of the authors \cite{BruinierOno}. They are meromorphic modular functions with a discriminant $D$ Heegner divisor, and their fields of definition are dictated by the Fourier coefficients of the $f_{[g]}$.

As the preceding discussion  illustrates, monstrous, umbral, and Thompson moonshine are (surprising) phenomena in which
a single infinite-dimensional graded $G$-module organizes information about divisors on products of modular curves that are indexed by the conjugacy classes of $G$. Moreover, the levels of these modular curves are (essentially) the orders of elements in these classes.
In the case of monstrous moonshine, the divisors are simple: they are of the form $cz-\infty$. In umbral and Thompson moonshine, we obtain Heegner divisors on $X(\Gamma_{[g]})$. 

The appearance of Heegner divisors recalls the seminal work of
Zagier \cite{Zag02} on traces of singular moduli on $X_0(1)$.
Loosely speaking, Zagier proved that the generating function for such traces in $D$-aspect can be
weight 3/2 weakly holomorphic modular forms. One of his motivations was to offer a classical perspective on
special cases of Borcherds' work \cite{Borcherds95}   on infinite product expansions of modular forms with Heegner divisor.

Although Zagier's paper has inspired too many papers to mention, we highlight an important note by
Gross \cite{Gross}. Gross observed that these types of theorems could be recast in terms of
  {\it generalized Jacobians} with cuspidal moduli. 
In particular, the generalized Jacobian of $X_0(1)$ with respect to the cuspidal divisor $2(\infty)$
is isomorphic to the additive group, and so the sum of the conjugates of Heegner
points in the generalized Jacobian is equal to the trace of their modular invariants.

Here we adopt this perspective.  Although we do not directly apply these results in this work, our view is that the McKay--Thompson series presented here should be  viewed in this way, as generating functions for traces of singular moduli and as functionals on Heegner divisors.
This interpretation is an extension of the celebrated theorem of Gross--Kohnen--Zagier \cite{GKZ} which asserts that the generating function for Heegner divisors on $X_0(N)$
are weight 3/2 cusp forms with values in the Jacobian of $X_0(N)$. By work of Waldspurger \cite{W1,W2} this earlier theorem can be thought of as a result on central
critical values of quadratic twists of weight 2 modular $L$-functions.

\subsection{Main Results}

In view of these developments, it is natural to seek weight 3/2 moonshine. One can loosely think of this as the moonshine obtained by summing weight 1/2 moonshine in $D$-aspect (e.g. umbral and Thompson moonshine), where the resulting
McKay--Thompson series are generating functions for the arithmetic of Heegner divisors.
 Namely, we seek moonshine of the form
$$
V^G=\bigoplus_m V^G_m\ \  \overset{\rm moonshine}{\xrightarrow{\hspace{1in}}} \ \ ( f_{[g]})\in 
\bigoplus\limits_{[g]\in \Conj(G)} H_{\frac{3}{2}}(\Gamma_{[g]})\otimes \Jac(X(\Gamma_{[g]})),
$$
where $\Jac(X(\Gamma_{[g]})$ denotes a suitable generalized Jacobian of $X(\Gamma_{[g]})$.
In such moonshine, the $f_{[g]}$ will be generating functions for suitable functionals over Heegner divisors. 
Their coefficients will be sums of class numbers, traces of singular moduli, and square-roots of central critical values of $L$-functions of quadratic twists of weight 2 modular forms.

Here we establish the first
example of moonshine of this type, and it is pleasing that pariah sporadic groups appear. We prove moonshine for  the O'Nan group $\ON$, a group discovered in 1973
as part of the flurry of activity related to the classification of finite simple groups \cite{Onan} and shown not to be involved in the monster by Griess \cite[Lemma 14.5]{Griess82}. This group was first constructed by Sims (cf. \cite[p. 421]{Onan}), and Ryba \cite{Ryba} later gave an alternative construction. It has order $\# \ON=2^9 \cdot 3^4\cdot 5\cdot 7^3 \cdot 11\cdot 19 \cdot 31$, and it has 30 conjugacy classes. It contains the first Janko group $J_1$, also not involved in the monster \cite{WilsonJ1}, as a subgroup.

\begin{theorem}\label{thm:Exists} There is an infinite-dimensional virtual graded $\ON$-module
$$
W:=\bigoplus_{0<m\equiv 0, 3\pmod{4}} W_m
$$
and weight 3/2 modular forms $\left \{F_{1A}, F_{2A},\dots, F_{31A}, F_{31B}\right\}$, one for each conjugacy class, with the property that
$$
F_{[g]}(\tau)= -q^{-4}+2+\sum_{0<m\equiv 0, 3\pmod{4}} \tr( g|W_m)q^m.
$$
Moreover, each $F_{[g]}$ is on the group $\Gamma_0(4o(g))$, with a non-trivial character in case $o(g)=16$, and satisfies the Kohnen plus space condition.
\end{theorem}

\begin{remark}
There is an alternative to \Cref{thm:Exists} in which the $F_{[g]}$ have trivial characters for all $g$, but are mock modular for $o(g)=16$. That formulation featured in an earlier version of this work. The present statement is motivated by cohomological considerations and related structures in the representation theory of vertex operator algebras, as we explain in more detail in \Cref{secProofclass}. We also characterize the $F_{[g]}$ precisely in \Cref{secProofclass}.
\end{remark}

\begin{remark}
In other prominent examples of moonshine (e.g. monstrous \cite{CN79} and umbral \cite{UM,MUM} moonshine) the McKay--Thompson series of a group element $g$ is a modular form (essentially) of level $o(g)$, but in this work the McKay--Thompson series $F_{[g]}$ have level $4o(g)$. This 
anomaly can be resolved by repackaging the $F_{[g]}$ as Jacobi forms as follows. 
For 
$g\in \ON$
set 
\begin{gather}\label{eqn:phig}
\varphi_{[g]}(\tau,z):=F_{[g],0}(\tfrac{\tau}{4})\theta_{1,0}(\tau,z)+F_{[g],1}(\tfrac{\tau}{4})\theta_{1,1}(\tau,z)
\end{gather} 
where $F_{[g],r}(\tau):=\sum_{m\equiv r\text{ mod } 2}\tr(g|W_m)q^m$ and $\theta_{1,r}(\tau,z):=\sum_{n\equiv r\text{ mod } 2}e^{2\pi i nz} q^{\frac{n^2}4}$. 
Then $\varphi_{[g]}$ 
is a weakly holomorphic Jacobi form of weight $2$ and index $1$ on $\Gamma_0(o(g))$, with a non-trivial character in case $o(g)=16$. 
For the sake of simplicity we have chosen to formulate our results in terms of the scalar-valued modular forms $F_{[g]}$ in this work. However, we note that one advantage of the Jacobi form formulation is that it illuminates an analogue of the Hauptmodul property of monstrous moonshine. Namely, each $\varphi_{[g]}$ has the property that it is uniquely determined, up to a cusp form, by the condition that it has growth of a certain form (independent of $g$) near the infinite cusp of $\Gamma_0(o(g))$, and vanishes at all other cusps. This follows from the proof of \Cref{thm:class}. It may be compared to monstrous moonshine, in which the McKay--Thompson series are uniquely determined up to constant functions by an analogous condition, and to umbral and Thompson moonshine, in which the McKay--Thompson series are uniquely determined by such a condition up to theta series (although in the umbral case almost all the relevant spaces of theta series vanish; cf. \cite{mumcor}).  It is the appearance of cusp forms that allows us to connect the O'Nan group to elliptic curve arithmetic.
\end{remark}

\begin{remark}
The module $W$ is {\em virtual} in the sense that some irreducible representations 
of $\ON$ occur with negative multiplicity in $W_m$ for some $m$. 
The proof of Theorem~\ref{thm:Exists} will show that only non-negative multiplicities appear for $m\notin\{7,8,12,16\}$.
So in fact we can replace $W$ with a non-virtual module for a small cost, by adding suitable multiples of weight $3/2$ unary theta functions (i.e. sums of the form $\sum_{n\in \Z}n\epsilon(n)q^{\lambda n^2}$ where $\epsilon$ is an odd periodic function and $\lambda$ is a positive rational) to the McKay--Thompson series $F_{[g]}$. This changes the module structure of $W_m$ when $m=vd^2$ for $v\in \{7,8,12,16\}$, for certain integers $d$, but it does not effect the validity of our other three main results, Theorems \ref{ClassNumberCongruences}, \ref{CongruencesSelmer}, and \ref{USCorollary}, for $-D<-16$. The price for such an adjustment to $W$ is the property that the McKay--Thompson series attached to $[g]$ have level $4o(g)$. It is this property which motivates us to focus on the particular module $W$ that appears in Theorem \ref{thm:Exists}.
\end{remark}

The $F_{[g]}$ will turn out to be expressible in terms of traces of singular moduli for
Hauptmoduln (cf. \Cref{secSing}), class numbers, and central critical $L$-values of quadratic twists of weight 2 modular forms (cf. \Cref{Lvalues}).
The Hauptmoduln which arise 
are for the
 genus 0 modular curves 
\begin{equation}\label{Genus0}
\{ X_0(N):  \ N=1,\dots,8, 10, 12,16\} \ \ \cup\ \ \{X_0^+(N): \ 
N=11, 14,15, 16, 19, 20, 28, 31,32\},
\end{equation}
where $X_0^+(N)$ is the modular curve corresponding to the extension of $\Gamma_0(N)$
by all the level $N$ Atkin--Lehner involutions.

\begin{remark}
Purely for the sake of curiosity we mention that it follows from the description of the dimensions of the graded components $W_m$ in terms of traces of singular moduli (cf. \Cref{appSing}) that
\[\dim W_{163}=\frac 12(\alpha^2+\alpha-393768),\]
where 
\[\alpha=\left\lceil e^{\pi\sqrt{163}}\right\rceil=\lceil 262537412640768743.999999999999250072...\rceil\]
denotes the Ramanujan constant. (This number was actually already discovered and studied by Hermite in 1859 \cite{Hermite}.)
\end{remark}

\begin{remark}
From \Cref{mults1,mults2,mults3} we see that $W_3$ is an irreducible $\ON$-module of dimension $26752$, and $W_4$ has three irreducible constituents, with dimensions $1$, $58311$ and $85064$. On the other hand the specialization $\varphi_{1A}(\tau,0)$ of \Cref{eqn:phig} is the derivative of the $J$ function, up to a scalar factor. 
This leads to the identity
$$
196884=5\cdot 1+ 2\cdot 26752 +58311+ 85064,
$$
where the summands on the right are dimensions of irreducible representations of $\ON$. 
Inspired by the moonshine module vertex operator algebra \cite{FLM} of Frenkel--Lepowsky--Meurman we may ask: is there a holomorphic vertex operator algebra with an action by $\ON$ that explains this coincidence? (See the second remark in \Cref{secProofclass} for some further related comments.)
\end{remark}

Armed with Theorem~\ref{thm:Exists}  and the explicit identities expressing the $F_{[g]}$ in terms
of singular moduli, class numbers and critical $L$-values,
it is natural to ask whether the infinite-dimensional $\ON$-module $W$ reveals arithmetic
information about the modular curves they organize, which include the positive genus curves
$$
\{X_0(11),\ X_0(14),\ X_0(15),\ X_0(19),\ X_0(20),\ X_0(28),\ X_0(31)\}
$$
related to the $X_0^+(N)$ in (\ref{Genus0}).
For example, are there interesting congruences modulo primes $p |\#\ON$ which relate the graded components $W_m$ to classical objects in number theory and arithmetic geometry?
This is indeed the case, and we now describe
 surprising congruences which relate graded dimensions and traces of
$W$ to class numbers and Selmer groups and Tate--Shafarevich groups of elliptic curves.

\begin{remark}
Suppose that $p$ is prime and $g_n$ (resp. $g_{np}$) are elements of $\ON$ with order $n$ (resp. $np$).
Then by \Cref{thm:Exists}, we have that $\tr(g_n|W_m)\equiv \tr(g_{np}|W_m)\pmod p$ for all $m$. 
In particular, if $o(g)=p$, then for all $m$ we have
$$
\dim W_m\equiv \tr(g|W_m)\pmod p.
$$
\end{remark}

The following theorem concerns congruences modulo small primes $p$ and ideal class groups of imaginary quadratic fields. Here and in the following,  we denote by $H(D)$ the Hurwitz class number of positive definite binary quadratic forms of discriminant $-D<0$ (cf. \Cref{secSing}). 

\begin{theorem}\label{ClassNumberCongruences}
Suppose that $-D<0$ is a fundamental discriminant. Then the following are true:

\begin{enumerate}
\item If $-D<-8$ is even and $g_2\in \ON$ has order 2, then
$$
\dim W_D\equiv \tr(g_2 |W_D)\equiv -24 H(D)\equiv 0\pmod{2^4}.
$$
\item If $p\in \{3, 5, 7\}$, $\left(\frac{-D}{p}\right)=-1$ and $g_p\in \ON$ has order $p$, then
$$
\dim W_D\equiv \tr(g_p |W_D)\equiv \begin{cases}
-24H(D)\pmod{3^2}\ \ \ \ \ \ &{\text {\rm if}}\ p=3,\\
-24H(D)\pmod{p}  \  \ \ \ \ \ &{\text {\rm if}}\ p=5, 7.
\end{cases}$$
\end{enumerate}
\end{theorem}
\begin{remark} Systematic congruences
 which assert for $\left(\frac{-D}{p}\right)=-1$ that
$$
\dim W_D\equiv -24 H(D)\pmod{p}
$$
do not seem to hold for  $p\geq 17$. However, this congruence holds for
 $p=13$, a bonus because $13\nmid \#\ON$.
\end{remark}

\begin{remark} 
As the proof of \Cref{ClassNumberCongruences} will reveal, it holds true that if $-D<-8$ is an even fundamental discriminant, then $H(D)$ is even, and $\dim W_D\equiv 0\pmod{2^4}$.
\end{remark}

In view of \Cref{ClassNumberCongruences}, it is natural to consider the 
primes $p= 11, 19$ and $31$ which also divide $\#\ON$. For these primes,  a refinement of the congruences
above is necessary.
In particular, for the primes 11 and 19 we obtain congruences which relate $\dim W_D$ to Selmer groups
and Tate--Shafarevich groups of elliptic curves (cf. \cite[Chapter X]{Silverman1}). 

Let $E/\Q$ be an elliptic curve given by 
$$
   E: \ \ \  y^2+a_1xy+a_3y=x^3+a_2x^2+a_4 x+a_6 
$$
where $a_1, a_2, a_3, a_4, a_6 \in \Z$. For a fundamental discriminant $D$, let
$E(D)$ denote its $D$-quadratic twist, and let $\rk(E(D))$ denote the Mordell--Weil rank of $E(D)$ over $\Q$.
The $\ON$-module  $W$ encodes deep information about the Selmer and Tate--Shafarevich groups of
the quadratic twists of elliptic curves with conductor 11, 14, 15, and 19.
To make this precise, suppose that $\ell$ is an odd prime. Then for each curve $E(D)$ we have the short exact sequence 
$$
  1\rightarrow E(D)/\ell E(D) \rightarrow \Sel(E(D))[\ell]
 \rightarrow \Sha(E(D))[\ell] \rightarrow 1,
$$
where $\Sel(E(D))[\ell]$ is the $\ell$-Selmer group of $E(D)$, and
$\Sha(E(D))[\ell]$ denotes the elements
of the Tate--Shafarevich group $\Sha(E(D))$ with order dividing $\ell$.

For $p=11$ and $19$, we let $E_p/\Q$ be the $\Gamma_0(p)$-optimal elliptic curves given by the Weierstrass models
\begin{displaymath}
\begin{split}
E_{11}&: \ \ y^2+y=x^3-x^2-10x-20,\\
E_{19}&: \ \ y^2+y=x^3+x^2-9x-15
\end{split}
\end{displaymath}
(cf. \cite[\href{http://www.lmfdb.org/EllipticCurve/Q/11.a2}{Elliptic Curve 11.a2}, \href{http://www.lmfdb.org/EllipticCurve/Q/19.a2}{Elliptic Curve 19.a2}]{LMFDB}).
We obtain the following congruence relating the graded dimension $\dim W_D$ to class numbers, and Selmer groups and Tate--Shafarevich groups of such twists.

\begin{theorem}\label{CongruencesSelmer}
Assume the Birch and Swinnerton-Dyer Conjecture. If $p=11$ or $19$ and $-D<0$ is a fundamental discriminant for which $\left(\frac{-D}{p}\right)=-1$, and $g_p\in \ON$ has order $p$, then the following are true.

\begin{enumerate}
\item We have that $\Sel(E_p(-D))[p]\neq \{0\}$ if and only if

$$
\dim W_D\equiv \tr(g_p | W_D) \equiv -24H(D)\pmod{p}.
$$

\item Suppose  that $L(E_p(-D),1)\neq 0$. Then we have that $\rk(E(-D))=0$. Moreover, we have $p|  \# \Sha(E_p(-D))$ if and only if
$$
\dim W_D\equiv \tr(g_p | W_D)\equiv -24H(D)\pmod{p}.
$$
\end{enumerate}
\end{theorem}
\begin{remark}
The claim about ranks in \Cref{CongruencesSelmer} (2) is unconditional thanks to the work of Kolyvagin \cite{Kolyvagin}. 
\end{remark}

\begin{remark}
By Goldfeld's famous conjecture on ranks of quadratic twists of elliptic curves
\cite{Goldfeld}, it turns out that the hypothesis in \Cref{CongruencesSelmer} (2)  is expected to
hold for $100\%$ of the $-D$ for which $\left(\frac{-D}{p}\right)=-1$. Therefore, for almost all such $-D$, we should have a test for determining the presence of order $p$ elements in these Tate--Shafarevich groups.
\end{remark}

\begin{remark} There is a more complicated congruence for the prime $p=31$. For fundamental discriminants $-D<0$ satisfying
$\left(\frac{-D}{31}\right)=-1$, we have that $\dim W_D\equiv \tr(g_{31}|W_D)\pmod{31}$ are related to the central critical
values of the $-D$ twists of the $L$-function for the genus 2 curve 
$$
C: \ \ y^2 + (x^3+x+1)y=x^5+x^4+x^3-x-1
$$
(cf. \cite[\href{http://www.lmfdb.org/Genus2Curve/Q/961/a/961/3}{Genus $2$ Curve 961.a.961.3}]{LMFDB}).
Its $L$-function arises from the
 two newforms in $S_2(\Gamma_0(31))$ which are Galois conjugates. Namely, if $\phi:=\frac{1+\sqrt{5}}{2}$
 then the two newforms are $f^{\sigma}$ and
$$
f(\tau):=\sum_{n=1}^{\infty}a(n)q^n=q+\phi q^2 -2\phi q^3 +(\phi-1)q^4+q^5-(2\phi+2)q^6+O(q^7),
$$
where $\sigma(\sqrt{5})=-\sqrt{5}$. If $p\nmid 31$ is prime, then the local $L$-factor $L_p(T)$ at $p$ is
$$
L_p(T):=(1-a(p)T+pT^2)(1-\sigma(a(p))T+pT^2).
$$
\end{remark}

\begin{remark} 
Apart from the claims about $\tr(g_{17}|W_D)$ (there are no elements of order 17 in $\ON$), \Cref{CongruencesSelmer} holds for $p=17$  as well. Namely, the congruences hold for $E_{17}$, the optimal
$\Gamma_0(17)$ elliptic curve over $\Q$ (cf.  \cite[\href{http://www.lmfdb.org/EllipticCurve/Q/17.a3}{Elliptic Curve 17.a3}]{LMFDB}) 
given by
$$
E_{17}: \ \ y^2+xy+y =x^3-x^2-x-14.
$$
\end{remark}

The two theorems on congruences above only pertain to the dimensions of the graded components of the $\ON$-module $W$.
We now turn to congruences for graded traces for elements of order 2 and 3. To this end, we
let $E_{14}$ and $E_{15}$ be the corresponding optimal elliptic curves over $\Q$
(cf. see \cite[\href{http://www.lmfdb.org/EllipticCurve/Q/14.a6}{Elliptic Curve 14.a6}, \href{http://www.lmfdb.org/EllipticCurve/Q/15.a5}{Elliptic Curve 15.a5}]{LMFDB})
 given by
\begin{displaymath}
\begin{split}
E_{14} & : \ \ y^2+xy+y = x^3+4x-6,\\
E_{15}& : \ \  y^2 +xy +y =x^3+x^2-10x-10.
\end{split}
\end{displaymath}
Using  work of Skinner and Skinner--Urban \cite{Skinner16, UrSk} related to the Iwasawa main conjectures for $\GL_2$, we obtain the following unconditional result.

\begin{theorem}\label{USCorollary}
Assume the notation above, and suppose that $N\in \{14, 15\}$. If $p$ is the unique prime $\geq 5$ dividing $N$, then let $\delta_p:=\frac{p-1}{2}$ and let $p':=N/p$. 
If $-D<0$ is a fundamental discriminant for which $\left(\frac{-D}{p}\right)=-1$ and $\left(\frac{-D}{p'}\right)=1$, then the following are true.
\begin{enumerate}
\item  We have that $\Sel(E_{N}(-D))[p]\neq \{0\}$ if and only if
$$
\tr(g_{p'} | W_D)\equiv \tr(g_{N} | W_D)\equiv \delta_p \cdot (H(D)-\delta_p H^{(p')}(D))\pmod{p}.
$$ 
\item Suppose  that $L(E_{N}(-D),1)\neq 0$. Then we have that $\rk(E(-D))=0$. Moreover, we have $p|\# \Sha(E_{N}(-D))$ if and only if
$$\tr(g_{p'} | W_D)\equiv \tr(g_{N} | W_D)\equiv \delta_p \cdot (H(D)-\delta_p H^{(p')}(D))\pmod{p}.
$$
\end{enumerate}
\end{theorem}

\begin{remark}
We note that \Cref{USCorollary} does not apply for $p=2$ (resp. $p=3$) when $N=14$ (resp. $N=15$).
In the case of $p=2$ the work of Skinner--Urban does not apply. For $p=3$ the connection between
graded traces and central values of Hasse-Weil $L$-functions does not hold. Namely, a critical hypothesis
due to Kohnen in terms of eigenvalues of Atkin--Lehner involutions fails (cf. \Cref{propKohnen}). 
\end{remark}

\begin{remark}
In view of the new results presented here, it is natural to wonder where one should look for further moonshine.
It seems likely that other sporadic groups will fall within the scope of weight 3/2 moonshine.  In another direction, one can ask about other half-integral weights. Also, it is natural to wonder if there are extensions of moonshine to Shimura curves and varieties. Are there
infinite-dimensional $G$-modules which organize the arithmetic of their divisors?
\end{remark}

\subsection{Methods} 
To prove \Cref{thm:Exists}, we employ the theory of Rademacher sums, harmonic Maass forms, and standard
facts about the representation theory of finite groups. Namely, we make use of the character table of $\ON$ (cf. \Cref{ct1}), and
the Schur orthogonality relations for group characters. In \Cref{secRademacher},  we first recall essential
facts about harmonic Maass forms and Rademacher sums.
In \Cref{secProofclass}, we prove a theorem which, using harmonic Maass forms, explicitly constructs weakly holomorphic weight 3/2 modular forms, one for each conjugacy
class of $\ON$. Furthermore, we establish that these modular forms have integer Fourier coefficients.
To complete the proof, we apply the Schur orthogonality relations to these functions to construct weight 3/2 modular forms
whose coefficients  encode the multiplicities of the irreducibles of the graded components of the alleged module $W$.
The proof is complete once it is established that these multiplicities are integral. 
Since the obstruction to integrality is bounded by group theoretical considerations, the proof of integrality follows 
by confirming
sufficiently many congruence relations among these forms. 
These calculations confirm that $W$ is a virtual module.
However, as mentioned earlier, it turns out that the multiplicities of each irreducible are non-negative in $W_m$ once
$m>16$. This claim follows from an analytic argument which involves bounding sums of Kloosterman sums.
These statements are proved in \Cref{secProofMoon}. In \Cref{secSing} we recall properties of singular moduli, and we interpret
the modular forms number theoretically in terms of singular moduli and class numbers and cusp forms.
We prove Theorems~\ref{ClassNumberCongruences}, \ref{CongruencesSelmer} and \ref{USCorollary} in \Cref{Applications}.
These proofs require the explicit formulas for the $F_{[g]}$, the results in \Cref{secSing}, and the work of Skinner--Urban
on the Birch and Swinnerton-Dyer Conjecture. We conclude the paper in \Cref{Examples} with numerical examples of some of these results.

\section*{Acknowledgements} \noindent
The authors thank Kathrin Bringmann, Michael Griffin, Dick Gross, Maryam Khaqan, Winfried Kohnen, Martin Raum, Jeremy Rouse, Jean-Pierre Serre and anonymous referees for helpful comments and corrections. The authors thank Theo Johnson-Freyd for communication regarding his joint work with David Treumann on the cohomology of the O'Nan group.
The authors thank Drew Sutherland for computing the elliptic curve invariants in \Cref{tblE14,tblE15}. 

\section{Rademacher Sums and Harmonic Maass Forms}\label{secRademacher}
Harmonic Maass forms are now a central topic in number theory. Their study originates from the work of Bruinier--Funke \cite{BF04} on geometric theta lifts and Zwegers' seminal work \cite{Zwegers} on Ramanujan's mock theta functions.
These realizations played a central role in the work of Bringmann and one of the authors on the Andrews--Dragonette Conjecture and
Dyson's partition ranks \cite{BO06, BO10}. For an overview on the subject of harmonic Maass forms and its applications in number theory and various other fields of mathematics, including mathematical physics, we refer the reader to \cite{BOOK, DMZ, Ono08, ZagierBourbaki}. 

Here, we briefly recall the essential facts about harmonic Maass forms that are required in this paper.
Namely, we recall
Rademacher sums, and we describe their projection to Kohnen's plus space.

\subsection{Rademacher Sums}

Here and throughout, we let $\tau=u+iv$, $u,v\in\R$, denote a variable in the upper half-plane $\HH$ and we use the shorthand $e(\alpha):=e^{2\pi i\alpha}$.

\begin{definition}
We call a smooth function $f:\HH\rightarrow\C$ a {harmonic Maass form}  of \emph{weight} $k\in\tfrac 12\Z$ and \emph{level} $N$ if the following conditions are satisfied:
\begin{enumerate}
\item We have $f|_k\gamma(\tau)=f(\tau)$ for all $\gamma\in\Gamma_0(N)$ and $\tau\in\HH$, where we define
\[f|_k\gamma (\tau):=\begin{cases} (c\tau+d)^{-k}f\left(\frac{a\tau+b}{c\tau+d}\right) & \text{if }k\in\Z \\
\left(\left(\frac{c}{d}\right)\eps_d\right)^{2k}\left(\sqrt{c\tau+d}\right)^{-2k}f\left(\frac{a\tau+b}{c\tau+d}\right) & \text{if }k\in\frac 12+\Z.
\end{cases}\]
with 
\[\eps_d:=\begin{cases} 1 & \text{if }d\equiv 1\pmod{4},\\ i & \text{if }d\equiv 3\pmod 4.\end{cases}\]
and where we assume $4|N$ if $k\notin\Z$.
\item The function $f$ is annihilated by the \emph{weight $k$ hyperbolic Laplacian},
\[\Delta_k f:=\left[-v^2\left(\frac{\partial^2}{\partial u^2}+\frac{\partial^2}{\partial v^2}\right)+ikv\left(\frac{\partial}{\partial u}+i\frac{\partial}{\partial v}\right)\right] f\equiv 0.\]
\item There is a polynomial $P(q^{-1})$ such that $f(\tau)-P(e^{-2\pi i\tau})=O(v^c)$ for some $c\in\R$  as $v\to\infty$. Analogous conditions are required at all cusps of $\Gamma_0(N)$.
\end{enumerate}
We denote the space of harmonic Maass forms of weight $k$ and level $N$  by $H_k(\Gamma_0(N))$.
\end{definition}

\begin{remark}
We note that condition (3) in the definition above differs from other definitions which occur commonly in the literature.
For example, {\it harmonic Maass forms with principal parts} are those forms for which the $O(v^c)$ bound is replaced by
$O(e^{-cv})$ for $c>0$. Namely, the harmonic Maass forms we consider here are permitted to have $0^{\rm th}$ Fourier coefficients which
are essentially powers of $v$.
\end{remark}

For the basic properties of these functions, we again refer to the literature mentioned above. We mention however the following lemmas. 

\begin{lemma}\label{lem:split}
Let $f\in H_{k}(\Gamma_0(N))$ be a harmonic Maass form of weight $k\neq 1$. Then there is a canonical splitting
\begin{equation}\label{eq:split}
f(\tau)=f^+(\tau)+f^-(\tau),
\end{equation}
where for some $m_0\in\Z$ we have the Fourier expansions
\[f^+(\tau):=\sum\limits_{n=m_0}^\infty c_f^+(n)q^n,\]
and
\[f^-(\tau):=\overline{c_f^-(0)}\frac{(4\pi v)^{1-k}}{k-1}+\sum\limits_{\substack{n=1}}^\infty \overline{c_f^-(n)}n^{k-1}\Gamma(1-k;4\pi nv)q^{-n},\]
where 
\[\Gamma(\alpha;x):=\int_x^\infty t^{\alpha}e^{-t}\frac{dt}{t}\]
 denotes the usual incomplete gamma function. 
\end{lemma}

\noindent
The $q$-series $f^{+}$ in \eqref{eq:split} is called the {\it holomorphic part} of the harmonic Maass form $f$. 

An important differential operator in the theory of harmonic Maass forms is the $\xi$-operator, a variation of the Maass lowering operator.
\begin{proposition}
The operator
\[\xi_k:H_k(\Gamma_0(N))\to M_{2-k}(\Gamma_0(N)),\; f\mapsto \xi_k(f):=2iy^k\overline{\frac{\partial f}{\partial \overline{\tau}}}\]
is a well-defined and surjective anti-linear map with kernel $M_{k}^!(\Gamma_0(N))$. 
\end{proposition} 
\emph{Mock modular forms} are the holomorphic parts of harmonic Maass forms. Any mock modular form has an associated modular form, called its \emph{shadow}, which is the image of its corresponding harmonic Maass form under the $\xi$-operator. A mock modular form with vanishing shadow is a (weakly holomorphic) modular form.

The next lemma seems to have been missed by the literature.

\begin{lemma}\label{lemcusp}
A harmonic Maass form whose holomorphic part vanishes at all cusps is a (holomorphic) cusp form.
\end{lemma}
\begin{proof}
This is a direct consequence of the properties of the Bruinier--Funke pairing (cf. Proposition 3.5 in \cite{BF04}).
\end{proof}

A convenient way to construct mock modular forms, which are holomorphic parts of harmonic Maass forms, is through \emph{Rademacher sums}. These were introduced by Rademacher in his work on coefficients of the $J$-function \cite{Rademacher}, and further developed in the context of moonshine mainly by Cheng, Frenkel and one of the authors \cite{CD11,CD12,DF}. 

Rademacher sums can be thought of as low weight analogues of Poincar\'e series. For a fixed level $N$ and some $K>0$, we define the set 
\[\Gamma_{K,K^2}(N):=\left\{\abcd\in\Gamma_0(N)\: : \: |c|<K\text{ and }|d|<K^2\right\}.\]
Given an integer $\mu$ 
we can use this to formally define the Rademacher sum
\[R^{[\mu]}_{k,N}(\tau):=\lim\limits_{K\to \infty} \sum_{\gamma\in\Gamma_\infty\setminus \Gamma_{K,K^2}(N)} q^\mu|_k\gamma\]
where as usual $\Gamma_\infty:=\left\{\pm\left(\begin{smallmatrix} 1 & n \\ 0 & 1 \end{smallmatrix}\right)\: : \: n \in\Z\right\}$ denotes the stabilizer of $\infty$ in $\Gamma_0(N)$.
If convergent, these sums define mock modular forms of the indicated weight, level and character.
Convergence for these series however is in general a delicate matter when the weight $k$ is between $0$ and $2$. We will be interested in these series when the weight is $k=\frac 32$ in which case it has been established in \cite[Section 5]{CD11} that they do converge (possibly using a certain regularization explained in loc. cit.) and define holomorphic functions on $\HH$. 

By construction, Rademacher sums are $1$-periodic and therefore have a Fourier expansion. It is given in terms of infinite sums of \emph{Kloosterman sums}
\begin{equation}\label{Kloost}
K_k(m,n,c):=\sideset{}{^*}\sum_{d\smod{c}}\left(\frac cd\right)\eps_d^{2k} e\left(\frac{m\overline{d}+nd}{c}\right)
\end{equation} 
weighted by Bessel functions. 
Here we have that $k\in\frac 12+\Z$, $c$ is divisible by $4$, the ${}^*$ at the sum indicates that it runs over primitive residue classes modulo $c$, and $\overline{d}$ denotes the multiplicative inverse of $d$ modulo $c$. Computing the Fourier expansion of a Rademacher sum is a standard computation, see for instance \cite[Section 3.1]{CD12} and \cite[Section 8.3]{Ono08}.

\begin{theorem}\label{FE}
Assuming locally uniform convergence, for $\mu\leq 0$ and $k\in\frac 12+\N$ and $4|N$, the Rademacher sum $R^{[\mu]}_{k,N}$ defines a mock modular form of weight $k$ for $\Gamma_0(N)$ whose shadow is given by a constant multiple of the Rademacher sum $R^{[-\mu]}_{2-k,N}$. Its Fourier expansion is given by
\[R^{[\mu]}_{k,N}(\tau)=q^\mu+\sum_{n=1}^\infty c_{k,N}^{[\mu]}(n)q^n,\]
where
\begin{equation}\label{eqFC}
c_{k,N}^{[\mu]}(n)= - 2 \pi  i^k   \left\vert \frac{n}{\mu}
\right\vert^{\frac{k-1}{2}}
  \sum_{\substack{c>0\\c\equiv 0\smod{N}}}
   \frac{K_{k}(\mu,n,c)}{c}\cdot
 I_{k-1}\!\left(\frac{4\pi\sqrt{|\mu n|}}{c}\right)
\end{equation}
for $\mu<0$ and
\begin{equation}\label{eqFC0}
c_{k,N}^{[0]}(n)= (-2\pi i)^k \frac{n^{k-1}}{\Gamma(k)}\sum_{\substack{c>0\\c\equiv 0\smod{N}}}\frac{K_{k}(0,n,c)}{c^k}.
\end{equation}

\noindent
The completion $\widehat{R^{[\mu]}_{k,N}}$ of $R^{[\mu]}_{k,N}$ to a harmonic Maass form has a pole of order $\mu$ at the cusp $\infty$ and vanishes at all other cusps.
\end{theorem}

\begin{remark}
One can also consider Rademacher sums of weights $\leq 1/2$, which are the main subject of \cite{CD11} and play a crucial rule in both umbral and Thompson moonshine. The formulas look very similar in those cases, but since they are not needed, we omit them here.
\end{remark}

\subsection{Kohnen's Plus Space}
In \cite{Kohnen85}, Kohnen introduced the notion of the so-called plus space, a natural subspace of weight $k+\frac 12$ cusp forms for $\Gamma_0(4N)$ which is isomorphic via the Shimura correspondence to the space of weight $2k$ cusp forms of level $N$ as a Hecke module, provided that $N$ is odd and square-free. This space is easily characterized via Fourier expansions. Namely, it consists of all forms in $S_{k+\frac 12}(\Gamma_0(4N))$ (or, by extension, $M_{k+\frac 12}^!(\Gamma_0(4N))$ and also $H_{k+\frac 12}(\Gamma_0(4N))$) whose Fourier coefficients are supported on exponents $n$ with  $n\equiv 0,(-1)^k\pmod 4$.  There is a natural projection operator
\[|\pr: S_{k+\frac 12}(\Gamma_0(4N))\rightarrow S_{k+\frac 12}^+(\Gamma_0(4N))\]
for $N$ odd given in terms of slash operators (see loc. cit.), which extends to spaces of weakly holomorphic modular forms and harmonic Maass forms. The action of this projection operator on principal parts of harmonic Maass forms is described in the following lemma (cf. Lemma 2.9 in \cite{GM16}).
\begin{lemma}\label{lemplus}
Let $N$ be odd and let $f\in H_{k+\frac 12}(\Gamma_0(4N))$ for some $k\in\N_0$. Suppose that 
\[f^+(\tau)=q^{-m}+\sum_{n=0}^\infty a_nq^n\]
for some $m>0$ with $-m\equiv 0,(-1)^k\pmod 4$, and suppose also that $f$ has a non-vanishing principal part only at the cusp $\infty$ and is bounded at the other cusps of $\Gamma_0(4N)$. Then the projection $f|\pr$ of $f$ to the plus space has a pole of order $m$ at $\infty$ and has a pole of order $\frac m4$ either at the cusp $\frac 1N$ if $m\equiv 0\pmod 4$, or at the cusp $\frac1{2N}$ if $-m\equiv(-1)^k\pmod 4$, and is bounded at all other cusps.
\end{lemma}

For the purpose of this paper, we are particularly interested in the Fourier expansions of weight $3/2$ Rademacher sums projected to the plus space (see the following section). Convergence for these follows along the same lines as in \cite[Section 5]{CD11}. The following proposition gives their Fourier expansion explicitly.

\begin{proposition}\label{FEplus}
Consider the Rademacher sum $R^{[\mu]}_{\frac 32,4N}$ for $\mu\leq 0$ such that $\mu\equiv 0,3\pmod 4$ and $N$ odd. Then we have that
\[R^{[\mu],+}_{\frac 32,4N}(\tau):=\left(R^{[\mu]}_{\frac 32,4N}|\pr\right) (\tau)=q^\mu+\sum_{\substack{n>0\\n\equiv 0,3\smod{4}}} c_{\frac 32,4N}^{[\mu],+}(n)q^n,\]
where we have
\begin{equation}\label{eqFEplus}
c_{\frac 32,4N}^{[\mu],+}(n)=\kappa(\mu,n)\sum_{c=1}^\infty \left(1+\delta_{odd}(Nc)\right)K_\frac 32 (\mu,n,4Nc)\cdot \mathcal{I}(\mu,n,4Nc),
\end{equation}
with
\begin{equation}\label{eqkappa}
\kappa(\mu,n):=\begin{cases} 2\pi e\left(-\frac 38\right) & \text{if } \mu=0,\\
2\pi e\left(-\frac 38\right)(n/|\mu|)^\frac 14 & \text{otherwise,}
\end{cases}
\end{equation}
\begin{equation}\label{deltaodd}
\delta_{odd}(n):=\begin{cases} 1 & \text{if }n\text{ is odd,} \\ 0 & \text{otherwise,}\end{cases}
\end{equation}
and
\begin{equation}\label{eqcalI}
\mathcal{I}(\mu,n,c):=\begin{cases}
\displaystyle 
\frac{(2\pi n)^\frac 12}{c^\frac 32\Gamma(3/2)} & \text{if }\mu=0,\\
 & \\
\displaystyle 
\frac{I_{\frac 12}\left(\frac{4\pi\sqrt{|\mu n|}}{c}\right)}{c} & \text{otherwise.}
\end{cases}
\end{equation}
\end{proposition}

The following proposition shows that the vanishing of Kloosterman sums automatically forces certain even level Rademacher sums to be in the plus space. 
\begin{proposition}\label{lemKloost}
The Rademacher sum $R^{[\mu]}_{\frac 32,4N}$ is automatically in the plus space if $N$ is even and $\mu\equiv 0,3\pmod 4$. Moreover, if $N,\mu\equiv0\pmod 4$, then the Fourier coefficients of $R^{[\mu]}_{\frac 32,4N}$ are supported on exponents divisible by $4$. \end{proposition}

\begin{proof}
We begin by noting that
if $c$ is divisible by $8$, then the Kloosterman sum $K(m,n,c)$ in \eqref{Kloost} vanishes unless $m-n\equiv 0,3\pmod{4}$. If $c$ is divisible by $16$, the same sum vanishes unless $m\equiv n\pmod 4$.
Therefore, the claim follows from Theorem~\ref{FE}.
\end{proof}

\begin{remark}
This is an easy restatement (and slight correction) of \cite[Lemma 2.10]{GM16}.
\end{remark}
\begin{remark}
\Cref{lemKloost} actually follows from the splitting properties of the Weil representation, which is a stronger statement than the vanishing of Kloosterman sums we employed in the proof. However, since we don't use the language of vector-valued modular forms in this paper we use the above more elementary argument. 
\end{remark}
\begin{remark}
We note that the formulas in \Cref{FEplus} also hold for even $N$ if one defines the projection operator $\pr$ for even levels as a suitable sieving operator, which one easily sees by a comparison to \Cref{FE}.
\end{remark}

\section{The Relevant Modular Forms}\label{secProofclass}

Here we use the results from the previous section to realize 
the McKay--Thompson series for  the $\ON$-module $W$ whose existence shall be proved later.
The main result here is the following theorem. To state it we define a character $\rho_{[g]}:\Gamma_0(4o(g))\to \C^*$ for each conjugacy class $[g]$ of $\ON$ by setting $\rho_{[g]}(\begin{smallmatrix}*&*\\c&*\end{smallmatrix}):=(-1)^{\frac{c}{128}}$ when $o(g)=16$, and letting $\rho_{[g]}$ be trivial otherwise.

\begin{theorem}\label{thm:class}
Assuming the notation above, the following are true.
\begin{enumerate}
\item For every conjugacy class $[g]$ of $\ON$ there is a unique weakly holomorphic modular form 
\begin{equation}
F_{[g]}(\tau)=-q^{-4}+2+\sum_{n=1} a_{[g]}(n)q^n
\end{equation}
of weight $3/2$ for the group $\Gamma_0(4o(g))$, with character $\rho_{[g]}$, satisfying the following conditions:
\begin{enumerate}
\item $F_{[g]}(\tau)$ lies in the Kohnen plus space, i.e., $a_{[g]}(n)=0$ if $n\equiv 1,2\pmod 4$.
\item $F_{[g]}(\tau)$ has a pole of order $4$ at the cusp $\infty$, a pole of order $\frac 14$ at the cusp $\frac 1{o(g)}$ if $o(g)$ is odd (as forced by the projection to the plus space, see \Cref{lemplus}), and vanishes at all other cusps.
\item We have $a_{[g]}(3)=\chi_7(g)$, and $a_{[g]}(4)=\chi_1(g)+\chi_{12}(g)+\chi_{18}(g)$, and $a_{[g]}(7)$ as given in \Crefrange{mults1}{mults3}, where $\chi_j$, for $j=1,...,30$, denotes the $j^{\rm th}$ irreducible character of $\ON$ as given in \Cref{ct1}. 
\end{enumerate}
\item The function $F_{[g]}(\tau)$ above has integer Fourier coefficients. 
\end{enumerate}
\end{theorem}

\begin{remark}
One can also give a more intrinsic description of the conditions in part (c) above. The proof of the theorem will show that $F_{[g]}$ is already determined by conditions (a) and (b) in part (1), for the $19$ conjugacy classes $[g]$ such that $o(g)\notin\{11,14,15,16,19,28,31\}$. For the remaining conjugacy classes we remark that whenever a prime $p$ divides $o(g)$, we need the congruence 
\begin{equation}\label{cong2}
a_{[g]}(n)\equiv a_{[g']}(n)\pmod{p}
\end{equation}
where $o(g')=o(g)/p$ in order for these to be generalized characters for $\ON$. Whenever one can choose
the coefficient $a_{[g]}(n)$ for the function $F_{[g]}$---which turns out to be the case for $n=3$ and $o(g)\in\{11,15,16\}$, for $n=4$ for $o(g)\in\{14,19\}$ and for $n\in\{4,7\}$ for $o(g)\in\{28,31\}$---we pick the least integer in absolute value satisfying \eqref{cong2} for all primes $p| o(g)$.
\end{remark}

\begin{remark}
The mod $2$ cohomology of the O'Nan group was computed by Adem--Milgram \cite{AM}. Using this, Johnson-Freyd--Treumann determined \cite{JFT} that $H^4(\ON,\Z)$ is cyclic of order $8$. Furthermore, they have explained to us \cite{Theo} that there is an element whose image under the restriction map $H^4(\ON,\Z)\to H^4(\langle g\rangle,\Z)$ is zero unless $o(g)=16$, in which case it is the element of order $2$ in $H^4(\Z/16\Z,\Z)\simeq \Z/16\Z$. The significance of this is that if $V$ is a holomorphic vertex operator algebra with an action by a finite group $G$ then it is conjectured that the ($G$-twisted) representation theory of $V$, including the modularity of its associated trace functions, is controlled by an element of $H^4(G,\Z)\simeq H^3(G,U(1))$. In particular, an element of $H^4(G,\Z)$ 
associates a multiplier system on $\Gamma_0(o(g))$ to each $g\in G$. The above statements about $H^4(\ON,\Z)$ imply that there is an element that associates the trivial multiplier to $\Gamma_0(o(g))$ for all $g\in \ON$ except those with $o(g)=16$, and in the latter case the multiplier arising is the order two character on $\Gamma_0(16)$ with kernel $\Gamma_0(32)$. That is, the characters determined by this element of $H^4(\ON,\Z)$
are exactly those that are satisfied by the Jacobi forms $\varphi_{[g]}$ of \Cref{eqn:phig}, and this is compatible with the existence of a holomorphic vertex operator algebra that realizes these functions, and hence also the $F_{[g]}$, and the $\ON$-module of \Cref{thm:Exists}. 
We refer to \S2 of \cite{Gannon}, and references therein, for more on the relationship between $H^3(G,U(1))$, modular forms and vertex operator algebra. We refer to \S3.2 of \cite{EG} for a recent account of the aforementioned conjecture on vertex operator algebras.
\end{remark}

\begin{remark}
The O'Nan group has a non-split extension $4^3\cdot \GL_3(2)$ as a subgroup, whilst the sporadic simple Higman--Sims group contains a splitting extension $4^3:\GL_3(2)$. In both cases the mentioned subgroups contain Sylow $2$-subgroups, so can be used to detect the existence, or not, of $2$-power elements in the cohomology of the corresponding simple groups. There may be something to be gained from a comparison of these groups, especially in light of the fact that the Higman--Sims group inherits a natural counterpart to monstrous moonshine by virtue of generalized moonshine, and the fact that it appears in the centralizers of suitable elements of order $5$ in the monster. We thank an anonymous referee for offering this observation. We also thank the referee for pointing out that the extensions of $\GL_3(2)$ by $4^3$ were studied by Alperin \cite{Alperin} and Griess \cite{Griess85}. Results on the cohomology of the Higman--Sims group can be found in \cite{ACKM} and \cite{JFT}.
\end{remark}

\begin{proof}[Proof of \Cref{thm:class}] Let $g\in \ON$ be any element with $o(g)\neq 16$. Then the difference of Rademacher sums
\[-R^{[-4]}_{\frac 32,4o(g)}(\tau)+2R^{[0]}_{\frac 32,4o(g)}(\tau)\]
of level $4o(g)$ is a mock modular form with the correct principal part at infinity and vanishes at all other cusps by \Cref{FE}. If $o(g)$ is even, then we know from \Cref{lemKloost} that this function is in the plus space. If on the other hand, $o(g)$ is odd, then we use the projection operator $|\operatorname{pr}$ to map it into the plus space, which by \Cref{lemplus} introduces an additional pole of order $\frac 14$ at the cusp $\frac{1}{o(g)}$. This establishes the existence of a function 
\[\widetilde F_{[g]}(\tau):=-R^{[-4],+}_{\frac 32,4o(g)}(\tau)+2R^{[0],+}_{\frac 32,4o(g)}(\tau)\]
satisfying properties (a) and (b) in \Cref{thm:class} (1) for $o(g)\neq 16$. 
To achieve this much for $o(g)=16$ we use
\[\widetilde F_{[g]}(\tau):=-(2R^{[-4],+}_{\frac 32,128}(\tau)-R^{[-4],+}_{\frac 32,64}(\tau))+2(2R^{[0],+}_{\frac 32,128}(\tau)-R^{[0],+}_{\frac 32,64}(\tau))\]
since $\rho_{[g]}$ in this case is trivial on $\Gamma_0(128)$, and $-1$ on $\Gamma_0(64)\setminus \Gamma_0(128)$. 

By \Cref{lemcusp} we see that the above properties determine a mock modular form uniquely up to cusp forms. Unless $o(g)\in\{11,14,15,16,19,28,31\}$ there are no cusp forms of weight $3/2$ in the plus spaces with the required characters, so one checks directly that in all those cases condition (c) is satisfied. In the remaining cases, condition (c) uniquely determines the contribution from cusp forms, because, as one can check using standard computer algebra systems
(the authors used \textsc{Magma} \cite{Magma} and \textsc{PARI} \cite{pari}), any weight $3/2$ cusp form of one of the given levels in the plus space with the relevant character is uniquely determined by the coefficients of $q^3,q^4,$ and $q^7$. 

We now show that the functions $F_{[g]}$ are actually all weakly holomorphic instead of just mock modular. First suppose that $o(g)$ is odd or $2||o(g)$. Then, because in those cases $o(g)$ is square-free, the shadow of $F_{[g]}(\tau)$ must be a multiple of 
$$\vartheta(\tau):=\sum_{n\in\Z}q^{n^2},$$
which follows from the Serre-Stark basis theorem \cite{SS77}. We compute Bruinier--Funke pairings (see Proposition 3.5 in \cite{BF04}) and find that
\[\{\widehat{R^{[-4],+}_{\frac 32,4o(g)}}(\tau),\vartheta(\tau)\}=2c\qquad\text{and}\qquad  \{\widehat{R^{[0],+}_{\frac 32,4o(g)}}(\tau),\vartheta(\tau)\}=c,\]
where $c$ is some constant. This shows that the shadow of the mock modular form $F_{[g]}(\tau)=-R^{[-4],+}_{\frac 32,4o(g)}(\tau)+2R^{[0],+}_{\frac 32,4o(g)}(\tau)$ is $0$, whence it is indeed a weakly holomorphic modular form.

If $o(g)$ is divisible by $4$ or $8$, but not $16$, the space of possible shadows is \emph{a priori} $2$-dimensional, generated by $\vartheta(\tau)$ and $\vartheta(4\tau)$, but \Cref{lemKloost} and the fact that the shadow of a Rademacher sum is again a Rademacher sum show that the shadow's Fourier coefficients must be supported on exponents divisible by $4$. 
So in fact, only multiplies of $\vartheta(4\tau)$ can occur as shadows and the same computation as above shows the claim in these cases. For $o(g)=16$ the space of possible shadows is a priori one-dimensional, spanned by $\vartheta(\tau)-\vartheta(4\tau)$, but this function is supported on odd exponents so is also ruled out by the Rademacher sum construction. 

It remains to show that the coefficients of the $F_{[g]}$ are all rational integers. This follows by checking finitely many coefficients and applying Sturm's Theorem \cite{Sturm}, in a manner directly similar to the proof of Proposition 3.2 in \cite{GM16}, for example. As we will also see in Section \ref{firstpart}, a possible bound up to which coefficients need to be checked to verify the claim is 225.
\end{proof}

\section{Proof of \Cref{thm:Exists}}\label{secProofMoon}

Here we prove that the weakly holomorphic modular forms given in \Cref{thm:class} are McKay--Thompson series for the infinite-dimensional $\ON$-module $W$. We begin by stating a refined form of \Cref{thm:Exists}.

\begin{theorem}\label{thm:moon}
There is an infinite-dimensional graded virtual $\ON$-module 
\[W=\bigoplus_{\substack{m=3 \\ m\equiv 0,3\smod 4}}^\infty W_m\]
such that we have
\[\tr(g | W_m)=a_{[g]}(m),\]
for all $m$. Moreover, $W_m$ is an honest $\ON$-module for $m\notin\{7,8,12,16\}$ (see \Crefrange{mults1}{mults3}).
\end{theorem}

We break down the proof of this theorem into separate pieces. Using the Schur orthogonality relations on the irreducible
representations 
of $\ON$ we construct weakly holomorphic modular forms of weight 3/2 whose coefficients are the multiplicities of the irreducible
components if and only if $W$ exists. Then the proof of \Cref{thm:moon} boils down to proving that these multiplicities are
integral for all $m$, and non-negative for $m\not \in \{7, 8, 12, 16\}$.
In \Cref{firstpart} we establish integrality, 
and in \Cref{secondpart} we establish the claim on non-negativity.

\subsection{Integrality of Multiplicities}\label{firstpart}
For every prime $p|\#\ON$ we find linear congruences among the alleged McKay--Thompson series. Here, we prove these, but first we note that their truth implies the following systematic congruences.

\begin{theorem}\label{theo:cong}
Let $g_j\in\ON$ of order $d_j$, $j=1,2$, with $d_2=p^c\cdot d_1$ for some prime number $p$ and $c\geq 1$. Then we have the congruence
\[F_{[g_1]}\equiv F_{[g_2]}\pmod{p}.\]
\end{theorem}

In \Cref{AppCong}, we list these congruences, which sometimes hold with higher prime power moduli than stated in \Cref{theo:cong}. Assuming their correctness for the moment, we can show integrality just as described in \cite{GM16}. For the convenience of the reader, we recall the method briefly. 

Let $\mathbf{C}\in\Z^{30\times\infty}$ denote the matrix formed by the coefficients of the functions $F_{[g]}(\tau)$ for each of the $30$ conjugacy classes of $\ON$ (in practice one uses a $30\times B$ matrix for some large $B$). Further denote by $\mathbf{X}\in\overline{\Q}^{30\times 30}$ the matrix whose rows are indexed by irreducible characters and whose columns are indexed by conjugacy classes of $\ON$, with
\[\mathbf{X}_{\chi,[g]}:=\frac{\overline{\chi(g)}}{\#C(g)},\]
where $C(g)$ denotes the centralizer of $g\in\ON$. By the first Schur orthogonality relation we see that the matrix 
\[\mathbf{m}:=\mathbf{X} \mathbf{C}\]
gives the multiplicities of each irreducible representation in the alleged virtual representation in \Cref{thm:moon}. Since there are repetitions among the rows of $\mathbf{C}$, because the functions $F_{[g]}(\tau)$ depend only on the order of elements in $[g]$, it does not have full rank, but by just deleting the repetitions it does turn out to have full rank, which is $18$. Let $\mathbf{N^*}\in\Z^{18\times 30}$ denote the matrix performing this operation and let $\mathbf{N}\in\Z^{30\times 18}$ be the matrix that undoes it, so that
\[\mathbf{m}=\mathbf{X}\mathbf{N}\mathbf{N^*}\mathbf{C}.\]
Now for each prime $p|\#\ON$, we can reduce the matrix $\mathbf{N^*}\mathbf{C}$ according to the aforementioned congruences as in \cite{GM16} by left-multiplying by a matrix $\mathbf{M}_p\in\Q^{18\times 18}$, which may be seen to have full rank. Hence we get
\[\mathbf{m}=(\mathbf{X}\mathbf{N}\mathbf{M}_p^{-1})\cdot(\mathbf{M}_p\mathbf{N^*}\mathbf{C}).\]
The congruences in \Cref{AppCong} ensure that the matrix $\mathbf{M}_p\mathbf{N^*}\mathbf{C}\in\Q^{18\times\infty}$ has all integer entries and one can check directly that the matrix $\mathbf{X}\mathbf{N}\mathbf{M}_p^{-1}\in\overline{\Q}^{30\times 18}$  has $p$-integral (rational) entries for every $p$. This shows that $\mathbf{m}$ has $p$-integral entries as well for each $p|\#\ON$, hence its entries must be integers, as claimed.

It remains to show the congruences. Since by \Cref{thm:class} all the functions $F_{[g]}$ 
are weakly holomorphic modular forms, we can prove all the congruences 
with standard techniques from the theory of modular forms. 
For example, we may multiply each of the congruences by the unique cusp form $g$ in $S_\frac{25}{2}^+(\Gamma_0(4))$ such that $g(\tau)=q^4+O(q^5)$ (which has integral coefficients), thereby reducing the problem to congruences among holomorphic modular forms of weight $14$. These can be checked in all cases using the Sturm bound \cite{Sturm}, which is at most 225 
in all cases.

\subsection{Positivity of Multiplicities}\label{secondpart}
Denote by $\mult_j(n)$ the multiplicity of the irreducible character $\chi_j$ of $\ON$ in the virtual module $W_n$ as in \Cref{thm:moon}, whose associated generalized character is given by the coefficients $a_{[g]}(n)$, cf. \Cref{thm:class}. Then the Schur orthogonality relations and the triangle inequality tell us that
\begin{equation}\label{eq:Schur}
\begin{aligned}
\mult_j(n)&=\sum_{[g]\subseteq \ON} \frac{1}{\# C(g)}a_{[g]}(n)\overline{\chi_j(g)}\\
		  &\geq \frac{|a_1(n)|}{\#\ON}\chi_j(1)-\sum_{[g]\neq 1A} \frac{|a_{[g]}(n)|}{\# C(g)}|\chi_j(g)|,
\end{aligned}
\end{equation}
where the summations run over conjugacy classes of $\ON$. Hence in order to show the eventual positivity of all $\mult_j(n)$, we want to establish explicit lower bounds on $a_{1A}(n)$, and upper bounds on $a_{[g]}(n)$ for $g\ne 1$. Recall that 
\[F_{[g]}(\tau)=-q^{-4}+2+\sum_{n=1}^\infty a_{[g]}(n)q^n=-R^{[-4],+}_{\frac 32,4o(g)}(\tau)+2R^{[0],+}_{\frac 32,4o(g)}(\tau)+\text{cusp form}\]
for $o(g)\neq 16$, and 
\[F_{[g]}(\tau)=-(2R^{[-4],+}_{\frac 32,128}(\tau)-R^{[-4],+}_{\frac 32,64}(\tau))+2(2R^{[0],+}_{\frac 32,128}(\tau)-R^{[0],+}_{\frac 32,64}(\tau))\]
for $o(g)=16$. (For $o(g)=16$ we have $\dim S^+_{\frac32}(\Gamma_0(4o(g)),\rho_{[g]})=1$ but condition (c) in \Cref{thm:class} rules out any cuspidal contribution to $F_{[g]}$. Cf. \Cref{appSing}.)

We bound each of the components individually, following the strategy already employed in \cite{DGO,Gannon,GM16}, which we sketch briefly for the convenience of the reader. Note however that in the cited papers, only the coefficients of one Rademacher sum had to be considered, since the corrections there were known to come from weight $\frac 12$ modular forms, whose coefficients are bounded. In our case the corrections can grow with $n$. 

Since the computations necessary to bound the contribution coming from the Rademacher sum $R^{[-4]}_{\frac 32,4N}$, which is obviously going to be the dominant part, have been carried out in detail in \cite{Gannon, GM16}, we omit them here. The idea is to use the known formula for the coefficients of the Rademacher sum in terms of infinite sums of Kloosterman sums weighted by $I$-Bessel functions, see \Cref{secRademacher}. One then splits this sum into three parts, a dominant part, an absolutely convergent remainder term and a value of a \emph{Selberg--Kloosterman zeta function}, the first two of which may be bounded by elementary means, and for the third, one uses Proposition 4.1 in \cite{GM16} (which we note is directly applicable to our situation).

\subsubsection{Bounding Coefficients of Rademacher Sums}
From \Cref{lem:trace0} below, we see that the $\mu=0$ Rademacher sum can be explicitly given in terms of generating functions of generalized Hurwitz class numbers $H^{(N)}(n)$ (see \Cref{secSing} for the definition). While strong bounds for class numbers are known (see for instance Chapter 23 in \cite{IK} and the references therein), they are usually not explicit. For our purposes, crude bounds on class numbers suffice. 
\begin{proposition}\label{prop:classbound}
For every $N\in\N$, $-D\leq -5$ a negative discriminant and $\eps>0$ we have
\[H^{(N)}(D)\leq [\SLZ:\Gamma_0(N)]c_\eps D^\eps\cdot \frac{\sqrt{D}}{2\pi}\left(1+\frac 12\log D\right),\]
where we can choose
\[c_\eps=\prod_{p<e^{\frac 1{2\eps}}} (2\eps p^{1/\log p-2\eps}\log p)^{-1}.\]
\end{proposition}
\begin{proof}
First we note that we trivially have the bound $H^{(N)}(D)\leq [\SLZ:\Gamma_0(N)]H^{(1)}(D)$ by definition. Now suppose for the moment that $D$ is a fundamental discriminant. Then Dirichlet's class number formula gives 
\[H(D)=\frac{\sqrt{D}}{2\pi}\cdot L\left(1,\left(\frac{-D}{\bullet}\right)\right).\]
Theorem 13.3 of Chapter 12 in \cite{Hua} tells us that for $D\geq 5$ we have the upper bound
\[L\left(1,\left(\frac{-D}{\bullet}\right)\right)<1+\frac 12\log D.\]
By \cite[pp. 73f.]{Zag81}, Dirichlet's formula is also valid for non-fundamental discriminants if only primitive forms are counted, so that we get the bound
\[H(D)\leq \tau_\square(D)\frac{\sqrt{D}}{2\pi}\left(1+\frac 12\log D\right),\]
where $\tau_\square(n)$ denotes the number of square divisors of $n$. Considering the prime factorisation of $D$, it is elementary to see that $\tau_\square(D)\leq c_\eps D^\eps$ for any $\eps>0$ and $c_\eps$ as claimed.
\end{proof}
This result together with \Cref{lem:trace0} gives a sufficient and explicit bound for the coefficients of the Rademacher sum $R^{[0],+}_{\frac 32, 4N}$. For the actual  computations we choose $\eps=\frac 18$, which yields $c_\eps\approx 10.6766$.

\subsubsection{Bounding Coefficients of Cusp Forms}\label{Lvalues}

For $g\in\ON$ with $$o(g)\in\{11,14,15,19,28,31\},$$ there are non-trivial cusp forms in $S_\frac 32^+(\Gamma_0(4o(g)))$ contributing to our modular forms $F_{[g]}$, see \Cref{appSing}. According to the Ramanujan--Petersson conjecture, the coefficients of these cusp forms should grow like $O(n^{\frac14+\eps})$ (for $n$ square-free). Unconditional bounds (again for square-free $n$) have been obtained by Iwaniec \cite{Iwan} for weights $\geq 5/2$ and Duke \cite{Duke} for weight $3/2$ (see also \cite{Schulze}). These bounds have one main disadvantage for our purposes, namely that the constants involved in them are not explicit or not computable. Here, we outline how to give completely explicit and computable, but very crude, estimates for the cusp form coefficients in question.

Let $P^{[m]}_{4N}$ denote the cuspidal Poincar\'e series of weight $3/2$ characterized by the \emph{Petersson coefficient formula},
\begin{equation}\label{eqPetersson}
\langle f, P^{[m]}_{4N}\rangle=\frac{b_f(m)}{[\SLZ:\Gamma_0(4N)]\sqrt{4m}}\quad\forall\ f(\tau)=\sum_{n=1}^\infty b_f(n)q^n\in S_\frac 32^+(\Gamma_0(4N)),
\end{equation}
where the Petersson inner product on $S_\frac 32^+(\Gamma_0(4N))$ is defined by the usual double integral
\[\langle f_1,f_2\rangle=\frac{1}{[\SLZ:\Gamma_0(4N)]}\int_{\Gamma_0(4N)\setminus\HH} f_1(\tau)\overline{f_2(\tau)}y^{\frac 32} \frac{du\,dv}{v^2}.\]

The Fourier coefficients of these Poincar{\'e} series are given in terms of infinite sums of Kloosterman sums times $J$-Bessel functions (see Proposition 4 in \cite{Kohnen85}), and essentially the same computation used to bound the coefficients of the Rademacher sums $R^{[-4]}_{\frac 32,4N}$ can be used here as well. It is then only necessary to express the cusp forms $\g^{(o(g))}$  (see again \Cref{appSing}) in terms of these Poincar\'e series, which is particularly easy in the cases where $o(g)\neq 31$ is odd, since in those cases, the space $S_\frac 32^+(\Gamma_0(4o(g)))$ is one-dimensional and $\g^{(o(g))}$ is a newform. Hence we have $\langle \g^{(o(g))},P^{[m]}_{4N}\rangle=\beta\langle \g^{(o(g))},\g^{(o(g))}\rangle$, where we choose $m$ to be the order of $\g^{(o(g))}$ at $\infty$. It therefore remains to compute the Petersson norm of the newform $\g^{(o(g))}$. This can be done by means of the following result due to Kohnen, which is an explicit version of Waldspurger's theorem (see Corollary 1 in \cite{Kohnen85}).
\begin{proposition}\label{propKohnen}
Let $N\in\N$ be odd and square-free, $f\in S_{k+\frac 12}^+(\Gamma_0(4N))$ be a newform and $F\in S_{2k}(\Gamma_0(N))$ the image of $f$ under the Shimura correspondence. For a prime $\ell|N$, let $w_\ell$ be the eigenvalue of $F$ under the Atkin--Lehner involution $W_\ell$ and choose a fundamental discriminant $D$ with $(-1)^kD>0$ and $\left(\frac{D}{\ell}\right)=w_\ell$ for all $\ell$. Then we have
\[\langle f,f\rangle=\frac{\langle F,F\rangle\pi^k}{2^{\omega(N)}(k-1)!|D|^{k-\frac 12}L(F,D;k)}\cdot |b_f(|D|)|^2,\]
where $L(F,D;s)$ denotes the twist of the newform $F$ by the quadratic character $\left(\frac D\bullet\right)$ and $\omega(N)$ denotes the number of distinct prime divisors of $N$.
\end{proposition}
Since the twisted $L$-series has a functional equation of the usual type, there are efficient methods to compute its values numerically. (The authors used the built-in intrinsics of \textsc{Magma} \cite{Magma}.) Computing the Petersson norm of $F$ is also possible to high accuracy, e.g. by using the well-known relationship (cf. \cite{Cremona, Zagier85})
\[\langle F,F\rangle=\frac{\vol(E)}{4\pi^2}\deg(\varphi_E),\]
where $F$ is the newform associated to the elliptic curve $E/\Q$, we denote the covolume of the period lattice of $E$ by $\vol(E)$, and use $\varphi_E$ for the modular parametrization of $E$. 
(Every elliptic curve $E$ we consider in this paper has $\deg(\varphi_E)=1$.) 
Alternatively, the Petersson norm of $F$ is computed for $N$ prime by Theorem 2 in \cite{Zagier85}.

\begin{remark}
Kohnen's result \Cref{propKohnen} has been extended to many situations, e.g. by Ueda and his collaborators \cite{Ueda1, Ueda2} to certain even levels and forms not in the plus-space (see in particular Corollary 1 in \cite{Manickametal}). So the above reasoning carries over to $o(g)\in\{14,28\}$, by noting that $\g^{(14)}$ and $\g^{(28)}$ both arise from the unique normalized cusp form in $S_\frac 32(\Gamma_0(28))$ (not in the plus space). The former may be obtained by applying sieve operators, the latter by applying the $V_4$-operator.
\end{remark}

\begin{remark} For $o(g)=31$, the above reasoning only needs to be modified to take into account that $\g^{(31)}$ is not a Hecke eigenform, but its decomposition into newforms is given in \Cref{appSing}. Using the fact that these newforms are orthogonal, the only difference becomes that one needs to take into account two Poincar\'e series instead of one.
\end{remark}

\noindent
Putting the estimates for the Rademacher sums $R^{[-4]}_{\frac 32,4N}$, $R^{[0]}_{\frac32, 4N}$, and the occuring cusp forms together and plugging them all into \eqref{eq:Schur}, one finds that the multiplicities are nonnegative as soon as $n\geq 113$ 
(the worst case occurs for the character $\chi_1$). Inspecting the remaining coefficients by computer then completes the proof of \Cref{thm:moon}.

\section{Traces of Singular Moduli}\label{secSing}

In this section, we discuss and recall some basic notation and facts about traces of singular moduli. Their study originates in seminal work by Zagier \cite{Zag02}, and has since been an important subject in number theory (cf. for instance \cite{BenLars,BO07,BF06,MillerPixton}, just to name a few). They appeared in connection with moonshine for the Thompson group in \cite{HarveyRayhaun}.

\subsection{Genus Zero Levels}
It is well-known that for $N\in \{1,2,3,4,5,6,7,8,10,12,16\}$ 
the modular curve $X_0(N)$ has genus $0$, so that in those cases, there is a Hauptmodul $J^{(N)}$. These Hauptmoduln are given explicitly in \Cref{tbl:HM0} in terms of the Dedekind eta function
$\eta(\tau):=q^{\frac{1}{24}}\prod_{n>0}(1-q^n)$
and the Eisenstein series $E_4(\tau):=1+240\sum_{n>0}n^3q^n(1-q^n)^{-1}$.
\begin{center}
\begin{table}[h]
\renewcommand{\arraystretch}{1.6}
\begin{tabular}{|c||c|c|c|c|c|c|}
\hline
$N$ & $1$ & $2$ & $3$ & $4$ & $5$ & $6$ \\
\hline
$J^{(N)}(\tau)$ & $\frac{E_4(\tau)^3}{\eta(\tau)^{24}}-744$ & $\frac{\eta(\tau)^{24}}{\eta(2\tau)^{24}}+24$ & $\frac{\eta(\tau)^{12}}{\eta(3\tau)^{12}}+12$ & $\frac{\eta(\tau)^{8}}{\eta(4\tau)^{8}}+8$ & $\frac{\eta(\tau)^{6}}{\eta(5\tau)^{6}}+6$ & $\frac{\eta(\tau)^{5}\eta(3\tau)}{\eta(2\tau)\eta(6\tau)^5}+5$    \\
\hline
\end{tabular}

\vspace{8pt}

\begin{tabular}{|c||c|c|c|c|c|}
\hline
$N$ & $7$ & $8$ & $10$ & $12$ & $16$ \\
\hline
$J^{(N)}(\tau)$ & $\frac{\eta(\tau)^{4}}{\eta(7\tau)^{4}}+4$ & $\frac{\eta(\tau)^{4}\eta(4\tau)^2}{\eta(2\tau)^{2}\eta(8\tau)^4}+4$ & $\frac{\eta(\tau)^{3}\eta(5\tau)}{\eta(2\tau)\eta(10\tau)^3}+3$ & $\frac{\eta(\tau)^{3}\eta(4\tau)\eta(6\tau)^2}{\eta(2\tau)^{2}\eta(3\tau)\eta(12\tau)^3}+3$ & $\frac{\eta(\tau)^{2}\eta(8\tau)}{\eta(2\tau)\eta(16\tau)^2}+2$  \\
\hline
\end{tabular}

\vspace{12pt}
\caption{Hauptmoduln for some $\Gamma_0(N)$}
\label{tbl:HM0}
\renewcommand{\arraystretch}{1.0}
\end{table}
\end{center}

To make use of these Hauptmoduln we require some notation. Denote by $\calQ_{-D}^{(N)}$ the set of positive definite quadratic forms $Q=ax^2+bxy+cy^2=:[a,b,c]$ of discriminant $-D=b^2-4ac<0$  such that $N|a$. It is well-known that $\Gamma_0(N)$ acts  on $\calQ_{-D}^{(N)}$ with finitely many orbits, which correspond to the so-called \emph{Heegner points} on the modular curve $X_0(N)$. For $Q=[a,b,c]\in\calQ_{-D}^{(N)}$,  we denote by $\tau_Q:=\frac{-b+i\sqrt{D}}{2a}$ the unique root of $Q(x,1)$ in $\HH$. For a function $f:\HH\rightarrow\C$ invariant under the action of $\Gamma_0(N)$ we then define the trace function
\begin{equation}\label{eqtrace}
\Tr_D^{(N)}(f):=\sum_{Q\in\calQ_{-D}^{(N)}/\Gamma_0(N)} \frac{f(\tau_Q)}{\omega^{(N)}(Q)},
\end{equation}
where $\omega^{(N)}(Q)=\frac{ 1}{2}\cdot\#\Stab_{\Gamma_0(N)}(Q)$. Further let 
\[\calH^{(N)}(\tau):=-\frac{[\SLZ:\Gamma_0(N)]}{12}+\sum_{\substack{D>0 \\ D\equiv 0,3\smod 4}} H^{(N)}(D)q^D\]
denote the generating function of the \emph{(generalized) Hurwitz class numbers of level $N$} which are defined as  $H^{(N)}(D):=\Tr_{D}^{(N)}(1)$. The special case of $N=1$ yields the classical Hurwitz class numbers $H^{(1)}(D):=H(D)$.

It is a straightforward consequence of Theorem 1.2 in \cite{MillerPixton}, analogous to Theorem 1.2 in \cite{BenLars}, that we can describe the Fourier coefficients of the Rademacher sums $R^{[-4],+}_{\frac 32,4N}$ as traces of the Hauptmoduln in \Cref{tbl:HM0}.

\begin{proposition}\label{lem:trace4}
Let $N\in\N$ such that $X_0(N)$ has genus $0$ and
\begin{equation}
\Tr_4^{(N)}(D):=\frac 12\left(\Tr_D^{(N)}(J^{(N)}_2)-\Tr_D^{(N/d)}(J^{(N/d)})\right),
\end{equation}
where $J^{(N)}_2=q^{-2}+O(q)$ is the unique modular function for $\Gamma_0(N)$ with this Fourier expansion at infinity and no poles anywhere else and $d:=\gcd(N,2)$. Then we have
\begin{equation}
\calT^{(N)}(\tau):=-q^{-4}+\sum_{\substack{D>0\\ D \equiv 0,3\smod 4}} \Tr_4^{(N)}(D)q^D=-R^{[-4],+}_{\frac 32,4N}(\tau)-\frac{c_2}{2}\calH^{(N)}(\tau)+\frac{c_1}{2}\calH^{(N/d)}(\tau)
\end{equation}
for certain rational numbers $c_1$ and $c_2$. In particular, the function $\calT^{(N)}$ has integer Fourier coefficients.
\end{proposition}

\begin{remark}
 It should be pointed out that Theorem 1.2 in \cite{MillerPixton} is only stated for odd levels, although the proof goes through for even levels as well.
 \end{remark}
 
\begin{remark}
The rational numbers $c_2$ and $c_1$ in \Cref{lem:trace4} are the constant terms of the weight $0$ Rademacher sums $R_{0,N}^{[-2]}$ and $R_{0,N/d}^{[-1]}$, respecitvely. For a proof of the rationality of these numbers see Lemma 3.2 in \cite{BenLars}.
\end{remark}

For the Rademacher sum $R^{[0],+}_{\frac 32,4N}$ we get the following.
\begin{proposition}\label{lem:trace0}
For $N\in\N$ 
we have
\[R^{[0],+}_{\frac 32,4N}=-\frac{12}{\varphi(N)}\sum_{d|N} \frac{d}{[\SLZ:\Gamma_0(d)]}\mu\left(\frac Nd\right)\calH^{(d)}\]
where $\mu$ and $\varphi$ denote the M\"obius function and Euler's totient function, respectively.
\end{proposition}
\begin{proof}
This follows from a straightforward modification of the proof of Theorem 1.2 in \cite{MillerPixton}.
\end{proof}
Note that the above \Cref{lem:trace0} is indeed valid for all $N$, not just those such that $X_0(N)$ has genus $0$.

Putting \Cref{lem:trace4,lem:trace0} together we obtain explicit descriptions of the functions $F_{[g]}$ in terms of singular moduli for $o(g)\in \{1,2,3,4,5,6,7,8,10,12\}$. These are given in \Cref{appSing}.

\subsection{Positive Genus Levels}
In the remaining cases, i.e. where $$o(g)\in\{11,14,15,16,19,20,28,31\},$$ 
our $F_{[g]}$ involve (cf. the proof of \Cref{thm:class}) Rademacher sums $R^{[-4],+}_{\frac32, 4N}$ where $N$ is such that $X_0(N)$ has positive genus. So there is no notion of a Hauptmodul there. However, it is known that for all these levels, the modular curve $X_0^+(N)$, being the quotient of $X_0(N)$ by all 
Atkin--Lehner involutions, does have genus $0$ (see e.g. \cite{FMN}). So there exists a Hauptmodul $J^{(N,+)}(\tau)$ for the corresponding group $\Gamma_0^+(N)$. See \Cref{tbl:HMplus} for these. There, $E_2(\tau):=1-24\sum_{n>0}nq^n(1-q^n)^{-1}$ is the quasimodular Eisenstein series, $f_{19}=q-2q^3+O(q^4)$ denotes the weight $2$ newform associated to the elliptic curve
\[E_{19}: y^2+y=x^3+x^2-9x-15\]
(\cite[\href{http://www.lmfdb.org/EllipticCurve/Q/19.a2}{Elliptic Curve 19.a2}]{LMFDB}), and $f_{31}=q+\frac{1+\sqrt{5}}{2}q^2+O(q^3)$ denotes the unique newform in $S_2(\Gamma_0(31))$ up to Galois conjugation (which is denoted by an exponent $\sigma$).

\begin{center}
\begin{table}[h]
\renewcommand{\arraystretch}{1.6}
\begin{tabular}{|c||c|c|c|}
\hline
$N$ & $11$ & $14$   \\
\hline
$J^{(N,+)}(\tau)$ & $-\frac{E_2(\tau)-11E_2(11\tau)}{10\eta(\tau)^2\eta(11\tau)^2}-\frac{22}{5}$  & $-\frac{E_2(\tau)+2E_2(2\tau)-7E_2(7\tau)-14E_2(14\tau)}{18\eta(\tau)\eta(2\tau)\eta(7\tau)\eta(14\tau)}-\frac{7}{3}$  \\
\hline
\end{tabular}

\vspace{8pt}

\begin{tabular}{|c||c|c|c|}
\hline
$N$ & $15$ & $16$ \\ 
\hline 
$J^{(N,+)}(\tau)$ 
& $-\frac{E_2(\tau)+3E_2(3\tau)-5E_2(5\tau)-15E_2(15\tau)}{16\eta(\tau)\eta(3\tau)\eta(5\tau)\eta(15\tau)}-\frac{5}{2}$  
& $\frac{\eta(2\tau)^6\eta(8\tau)^6}{\eta(\tau)^4\eta(4\tau)^4\eta(16\tau)^4}-4$
\\
\hline
\end{tabular}

\vspace{8pt}

\begin{tabular}{|c||c|c|c|c|}
\hline
$N$ & $19$ & $20$ & $28$  \\
\hline
$J^{(N,+)}(\tau)$  
& $-\frac{E_2(\tau)-19E_2(19\tau)}{18f_{19}(\tau)}-\frac{4}{3}$
& $\frac{\eta(2\tau)^8\eta(10\tau)^8}{\eta(\tau)^4\eta(4\tau)^4\eta(5\tau)^4\eta(20\tau)^4}-4$ 
& $\frac{\eta(2\tau)^6\eta(14\tau)^6}{\eta(\tau)^3\eta(4\tau)^3\eta(7\tau)^3\eta(28\tau)^3}-3$ \\
\hline
\end{tabular}

\vspace{8pt}
\begin{tabular}{|c||c|c|c|c|}
\hline
$N$ & $31$ & $32$\\
\hline
$J^{(N,+)}(\tau)$  
	& $\frac{\sqrt{5}(f_{31}(\tau)+f_{31}^\sigma(\tau))}{2(f_{31}(\tau)-f_{31}^\sigma(\tau))}-\frac 52$
	& $\frac{\eta(2\tau)^3\eta(16\tau)^3}{\eta(\tau)^2\eta(4\tau)\eta(8\tau)\eta(32\tau)^2}-2$			\\
\hline
\end{tabular}

\vspace{8pt}
\caption{Hauptmoduln for some $\Gamma_0^+(N)$}
\label{tbl:HMplus}
\renewcommand{\arraystretch}{1.0}
\end{table}
\end{center}

Armed with these Hauptmoduln we can now express the Fourier coefficients of all the remaining Rademacher sums $R^{[-4],+}_{\frac 32,4N}$ in terms of singular moduli of holomorphic modular functions and class numbers. 

\begin{proposition}\label{prop:trace4plus}
Let $N\in\N$ such that $X_0^+(N)$ has genus $0$ and define
\begin{equation}
\Tr_4^{(N,+)}(D):=\frac 12\left(\frac{1}{2^{\omega(N)}}\Tr_D^{(N)}\left(J^{(N,+)}_2\right)-\frac{1}{2^{\omega(N/d)}}\Tr_D^{(N/d)}\left(J^{(N/d,+)}\right)\right)
\end{equation}
where $J^{(N,+)}_2=q^{-2}+O(q)$ is the unique modular function for $\Gamma^+_0(N)$ with this Fourier expansion at infinity and no poles anywhere else and $d:=\gcd(N,2)$. Then we have
\begin{equation}
\calT^{(N,+)}(\tau):=-q^{-4}+\sum_{\substack{D>0 \\ D\equiv 0,3\smod 4}} \Tr_4^{(N,+)}(D)q^D=-R^{[-4],+}_{\frac 32,4N}(\tau)-\frac{c_2}{2}\calH^{(N)}(\tau)+\frac{c_1}{2}\calH^{(N/d)}(\tau)
\end{equation}
for some rational numbers $c_1$ and $c_2$, where $\omega(N)$ denotes the number of distinct prime factors of $N$. 
\end{proposition}
\begin{proof}
\Cref{lem:trace4} turns out to be valid for all $N$, if one replaces the Hauptmodul $J^{(N)}$ by the completed Rademacher sum $\widehat{R^{[-1]}_{0,N}}$, normalized so that its constant term is $0$ and $J^{(N)}_2$ by $\widehat{R^{[-2]}_{0,N}}$ with the same normalization, which is the original formulation in \cite{MillerPixton}. Note that these Rademacher sums coincide with the Hauptmoduln where applicable. Now we consider for $N'||N$, i.e. $\gcd(N',N/N')=1$, the Atkin--Lehner involution $W_{N'}$. These involutions map the set $\calQ_{-D}^{(N)}/\Gamma_0(N)$ bijectively to itself (see e.g. Section 1 of \cite{GKZ}), so for any $\Gamma_0(N)$-invariant function $f$ we have that $\Tr_D^{(N)}(f)=\Tr_D^{(N)}(f|W_{N'})$. Since there are exactly $2^{\omega(N)}$ Atkin--Lehner involutions of level $N$ this means that
\[\Tr_D^{(N)}(f)=\frac{1}{2^{\omega(N)}}\Tr_D^{(N)}(\tilde{f})\]
where $\tilde{f}:=\sum_{N'||N} f|W_{N'}$. The function $\tilde{f}$ is clearly $\Gamma^+_0(N)$-invariant. By checking the polar parts we conclude that if $f$ is $\widehat{R^{[-1]}_{0,N}}$ or $\widehat{R^{[-2]}_{0,N}}$ then $\tilde{f}$ has to coincide with $J^{(N,+)}$ or $J^{(N,+)}_2$, respectively, up to a rational additive constant. This proves the result.
\end{proof}

We now put \Cref{prop:trace4plus,lem:trace0} together to obtain explicit descriptions of the functions $F_{[g]}$ in terms of singular moduli for $o(g)\in \{11,14,15,19,28,31\}$. For $o(g)=16$ we use \Cref{lem:trace4} as well since $F_{[g]}$ involves Rademacher sums $R^{[\mu],+}_{\frac32,4N}$ for both $N=16$ and $N=32$ in this case. The resulting expressions are given in \Cref{appSing}.

\section{Number Theoretic Applications}\label{Applications}

In this section we prove the arithmetic applications of O'Nan moonshine given in \Crefrange{ClassNumberCongruences}{USCorollary}. All these proofs rely on the following easy observation.
\begin{lemma}\label{lemHeegner}
Let $N>1$ be an integers and $-D<0$ a discriminant which is not a square in $\Z/N\Z$. Then the set $\calQ_{-D}^{(N)}$ is empty. In particular, we have that $\Tr_{-D}^{(N)}(f)=H^{(N)}(D)=0$ for any $\Gamma_0(N)$-invariant function $f$.
\end{lemma}
\begin{proof}
A quadratic form $[a,b,c]\in\calQ_{-D}^{(N)}$ satisfies $-D=b^2-4ac$ and $N|a$, hence if $-D$ is not a square modulo $N$, there cannot be any such forms.
\end{proof}

\subsection{Proof of \Cref{ClassNumberCongruences}}

Suppose first that $p\in\{5,7\}$ and let $D$ be as in \Cref{ClassNumberCongruences}. The congruences in \Cref{AppCong} together with the identities in \Cref{appSing} imply the congruence
\[\dim(W_D)\equiv\tr(g_p|W_D)\equiv \Tr_4^{(p)}(D)-24H(D)+\alpha_pH^{(p)}(D)\pmod p\]
for some integer $\alpha_p$. By \Cref{lemHeegner}, the terms $\Tr_4^{(p)}(D)$ and $H^{(p)}(D)$ vanish for $D$ as required, proving the result. For $p=3$, one replaces the modulus above by $3^2$, making the congruence non-trivial.

For $p=2$, we note that there is a congruence between $\dim W_D$ and $\tr(g_2|W_D)$ modulo $2^{11}$ by \Cref{AppCong}. As one easily sees through a Sturm bound argument, we also have 
\[\dim W_D\equiv \tr(g_2|W_D)\equiv 0\pmod{16}\]
for $D\equiv 4,8\pmod {16}$, which is the case in particular when $-D<8$ is an even fundamental discriminant. The fact that for these $D$ the class number is even can be seen in various ways, for example by noting that by a famous theorem of Gauss and Hermite we have that $24H(D)=2r_3(D/4)$, where $r_3(n)$ is the number of representations of $n$ as the sum of three squares. Since $-D$  is fundamental, it follows that $D/4$ is square-free and hence is not the sum of three or just two equal squares. Through an easy case-by-case analysis one then finds that $r_3(D/4)$ is always divisible by $8$. Alternatively, one could also show the modular forms congruence 
\[\sum_{n\equiv 1,2\pmod 4} r_3(n)q^n\equiv 6\sum_{n=0}^\infty q^{(2n+1)^2}+4\sum_{n=0}^\infty q^{2(2n+1)^2}\pmod 8.\]
This completes the proof.

\subsection{Preliminaries on Elliptic Curves}

The proofs of \Cref{CongruencesSelmer,USCorollary} require a little preparation which we provide in this section.

One of the most important open problems in the theory of elliptic curves is the Birch and Swinnerton-Dyer Conjecture.
\begin{conjecture}\label{BSD}
Let $E/\Q$ be an elliptic curve. Then we have that
\begin{equation}\label{eqBSD}
\frac{L^{(r)}(E,1)}{r!\Omega_E}=\frac{\#\Sha(E)\cdot\Reg(E) \prod_\ell c_\ell(E)}{(\#E(\Q)_{tors})^2},
\end{equation}
where $r$ denotes the order of vanishing of $L(E,s)$ at $s=1$, which equals the Mordell--Weil rank of $E$, $\Omega_E$ is the real period of $E$, $\#\Sha(E)$ and $\Reg(E)$ denote the order of the Tate-Shafarevich group and the regulator of $E$, respectively, 
the $c_\ell(E)$ for prime $\ell$ are the Tamagawa numbers of $E$, and $\#E(\Q)_{tors}$ signifies the order 
of the torsion subgroup of the $\Q$-rational points of $E$.
\end{conjecture}

The weak Birch and Swinnerton-Dyer conjecture---that the order of vanishing of $L(E,s)$ at $s=1$ equals the rank of $E$---was established for curves of ranks $0$ and $1$ through work of Gross--Zagier \cite{GrZ86} and Kolyvagin \cite{Kolyvagin}. More recently, Bhargava--Shankar \cite{BhargavaShankar} proved, using Kolyvagin's theorem and the proof of the Iwasawa main conjectures for $\GL_2$ by Skinner--Urban \cite{UrSk} (among other deep results), that a positive proportion of all elliptic curves satisfy the weak Birch and Swinnerton-Dyer Conjecture. 

It is known that the left-hand side of \eqref{eqBSD} is always a rational number, see for instance \cite[Theorem 3.2]{Agashe1}. The following result shows that in certain situations, a local version of \Cref{BSD}, which is going to be sufficient for our purposes, holds.
\begin{theorem}[\cite{Skinner16}, Theorem C]\label{Skinner}
Let $E/\Q$ be an elliptic curve and $p\geq 3$ a prime of good ordinary or multiplicative reduction. Further assume that the $\Gal(\overline{\Q}/\Q)$-representation $E[p]$ is irreducible and that there exists a prime $p'\neq p$ at which $E$ has multiplicative reduction and $E[p]$ ramifies. If $L(E,1)\ne 0$, then we have that
\[\ord_p\left(\frac{L(E,1)}{\Omega_E}\right)=\ord_p\left(\#\Sha(E)\prod_\ell c_\ell(E)\right).\]
If $L(E,1)=0$, then we have $\Sel(E)[p]\neq\{0\}$.
\end{theorem}

We are especially interested in quadratic twists of elliptic curves. In this context, the following result by Agashe, giving the real period of such a twist, turns out to be very useful.

\begin{lemma}[\cite{Agashe2}, Lemma 2.1]\label{Agashe}
Let $E/\Q$ be an elliptic curve of conductor $N$ and let $-D<0$ be a fundamental discriminant coprime to $N$. Then we have that
\[\Omega_{E(-D)}=c_E\cdot c_\infty(E(-D))\cdot \omega_-(E)/\sqrt{D},\]
where $c_E$ denotes the Manin constant of $E$, $c_\infty(E(-D))$ denotes the number of components of $E(-D)$ over $\R$, and $\omega_-(E)$ denotes the second period of the period lattice of $E$.
\end{lemma}
\begin{remark}
The famous Manin Conjecture states that $c_E=1$. 
\end{remark}

Combining this with a theorem of Kohnen \cite{Kohnen85} (cf. \Cref{propKohnen}), we obtain the following.

\begin{lemma}\label{lemLRatio}
Let $E/\Q$ be an elliptic curve of odd, square-free conductor $N$ and let $-D<0$ be a fundamental discriminant satisfying $\left(\frac{-D}{\ell}\right)=w_\ell$, where $w_\ell$ denotes the eigenvalue of the newform $F_E\in S_2(\Gamma_0(N))$ associated to $E$ and the Atkin--Lehner involution $W_\ell$, $\ell|N$. Denote by $D_0$ the smallest such discriminant. Further let $f_E(\tau)=\sum_{n=3}^\infty b_E(n)q^n\in S_{\frac 32}^+(\Gamma_0(4N))$ be the weight $3/2$ cusp form associated to $F_E$ under the Shintani lift. For $p\geq 3$ prime we then have that
\[\ord_p\left(\frac{L(E(-D),1)}{\Omega_{E(-D)}}\right)=\ord_p\left(\frac{L(E(-D_0),1)}{\Omega_{E(-D_0)}}\right)+\ord_p\left(b_E(|D|)^2\right).\]
\end{lemma}
\begin{proof}
By combining \Cref{propKohnen} and \Cref{Agashe}, we find for the fundamental discriminants $-D<0$ as in the lemma that
\begin{equation}\label{eqLRatio}
\frac{L(E(-D),1)}{\Omega_{E(-D)}}=\frac{\pi\langle F,F\rangle}{c_E\cdot c_\infty(E(-D))2^{\omega(N)}\langle f,f\rangle\omega_-(E)}\cdot |b_E(D)|^2.
\end{equation}
We see that the only quantities in this formula depending on $D$ are $c_\infty(E(-D))$ and $b_E(D)$. Since the former is always either $1$ or $2$ and $p$ is odd, it doesn't affect the $p$-adic valuation at all, which proves the lemma.
\end{proof}
\begin{remark}
If the conductor $N$ is even but still square-free, the same result still holds along the same lines, using the remark following \Cref{propKohnen}. The exact formula in this case only differs from \eqref{eqLRatio} by a power of $2$, which doesn't affect the $p$-adic valuation. 
\end{remark}

\subsection{Proofs}
In this section, we prove \Cref{CongruencesSelmer,USCorollary}. The proofs  of both theorems are very similar in their main steps, so we combine them here.
\begin{proof}[Proof of \Cref{CongruencesSelmer,USCorollary}]
By applying the expressions for the relevant $F_{[g]}$ in terms of the traces of singular moduli, class numbers and weight $3/2$ cusp forms in \Cref{appSing}, and using the congruences in \Cref{AppCong}, we find that
\begin{align*}
\dim(W_D)&\equiv \tr(g_{11}|W_D)\equiv \Tr_4^{(11)}(D)-24H(D)+\alpha_{11}H^{(11)}(D)+\gamma_{11}b_{11}(D)\pmod{11},\\
\tr(g_2|W_D)&\equiv \tr(g_{14}|W_D)\equiv \Tr_4^{(14)}(D)+\delta_7(H(D)-\delta_7H^{(2)}(D))\\
& \ \ \ \ \ \ \ \ \ \ \ \ \  \qquad\qquad\qquad\quad +\alpha_7H^{(7)}(D)+\beta_7H^{(14)}(D)+\gamma_7b_{14}(D)\pmod{7},\\
\tr(g_3|W_D)&\equiv \tr(g_{15}|W_D)\equiv \Tr_4^{(15)}(D)+\delta_5(H(D)-\delta_5H^{(3)}(D))\\
& \ \ \ \ \ \ \ \ \ \ \ \ \ \qquad\qquad\qquad\quad +\alpha_5H^{(5)}(D)+\beta_5H^{(15)}(D)+\gamma_5b_{15}(D)\pmod{5},\\
\dim(W_D)&\equiv \tr(g_{19}|W_D)\equiv \Tr_4^{(19)}(D)-24H(D)+\alpha_{19}H^{(19)}(D)+\gamma_{19}b_{19}(D)\pmod{19},
\end{align*}
where $\delta_p=\frac{p-1}2$, $\alpha_p,\beta_p$ are some integers, $\gamma_p$ are $p$-adic units, and $b_N(D)$ denotes the $D^{\rm th}$ coefficient of the weight $3/2$ cusp form $\g^{(N)}$ specified in \Cref{appSing}. If $-D$ is a fundamental discriminant as specified in \Cref{CongruencesSelmer,USCorollary} respectively, then by \Cref{lemHeegner}, the terms $\Tr_4^{(N)}(D)$ as well as $H^{(p)}(D)$ and $H^{(N)}(D)$ above disappear. This shows that the class number congruences in our theorems hold if and only if $p$ divides the coefficient $b_N(D)$, i.e. if and only if $\ord_p\left(\frac{L(E_N(-D),1)}{\Omega_{E_N(-D)}}\right)>0$  by \Cref{lemLRatio}.
(A \textsc{Magma} computation reveals that the ratio $\frac{L(E_N(-D_0),1)}{\Omega_{E_N(-D_0)}}$ for the smallest possible $D_0$ is in each case a $p$-adic unit.)

Suppose for simplicity that $L(E_N(-D),1)\ne 0$. According to the Birch and Swinnerton-Dyer \Cref{BSD}, this implies that 
\[\ord_p(\#\Sha(E_N(-D))\prod_\ell c_\ell(E(-D)))>0,\]
so our theorems follow, conditionally on \Cref{BSD}, if the Tamagawa numbers $c_\ell(E(-D))$ are never divisible by $p$ in our cases. To establish this, we note (cf. \cite[Appendix C, Table 15.1]{Silverman1}) that for an elliptic curve $E/\Q$ we have that $p|c_\ell(E)$ if and only if the reduction type of $E$ at $\ell$ is $I_n$ with $p|n$, which means that $\ord_\ell(\Delta(E))=n$, where $\Delta(E)$ denotes the (minimal) discriminant of $E$. An inspection of Tate's algorithm for the computation of Tamagawa numbers and the well-known formulas for minimal discriminants from the Kraus--Laska algorithm reveals that in our case, because we are considering twists of elliptic curves by fundamental discriminants, all the Tamagawa numbers must be in $\{1,2,3,4\}$. The argumentation in the case $L(E(-D),1)=0$ is similar. This completes the proof of \Cref{CongruencesSelmer} for $N\in\{11,19\}$.

The truth of \Cref{USCorollary} does not depend on the Birch and Swinnerton-Dyer Conjecture, but rather on Skinner's \Cref{Skinner}. A lemma of Serre \cite[\S 2.8, Corollaire, p. 284]{Serre} shows that the Galois representations $E_{14}(-D)[7]$ and $E_{15}(-D)[5]$ are surjective and hence irreducible. Furthermore, it is immediate to check that $E_{14}(-D)$ (resp. $E_{15}(-D)$) has multiplicative reduction modulo $2$ (resp. $3$) and that $E_{14}(-D)[7]$ (resp. $E_{15}(-D)[5]$) ramifies there, so the conditions of \Cref{Skinner} are satisfied, completing the proof of \Cref{USCorollary}.
\end{proof}

\section{Examples}\label{Examples}

Here we offer some numerical examples which illustrate the congruences described in the introduction.

\subsection{Class Number Congruences}
Here we present some class number congruences that arise from \Cref{ClassNumberCongruences}.
Recall that this theorem offers congruences modulo 16, 9, 5, and 7 for certain fundamental discriminants $-D<0$ which
satisfy given congruence conditions. The three columns in \Crefrange{tblclass2}{tblclass7} are congruent, which illustrates the theorem.

\begin{center}
\begin{table}[h]
\begin{tabular}{|r|c|c|c|}
\hline
$D$ & $\dim W_D$ & $\tr(g_2|W_D)$ & $-24H(D)$ \\
\hline
\hline
20 & $798588584512\equiv 0\pmod{16}$ & $576\equiv 0\pmod{16}$ & $-48\equiv 0\pmod{16}$\\
24 & $116700...6880\equiv 0\pmod{16}$ & $-1088\equiv 0\pmod{16}$ & $-48\equiv 0\pmod{16}$ \\
40 & $905977...8912\equiv 0\pmod{16}$ & $-10304\equiv 0\pmod{16}$ & $-48\equiv 0\pmod{16}$\\
\hline
\end{tabular}
\caption{$p=2$}
\label{tblclass2}
\end{table}

\begin{table}[h]
\begin{tabular}{|r|c|c|c|}
\hline
$D$ & $\dim W_D$ & $\tr(g_3|W_D)$ & $-24H(D)$ \\
\hline
\hline
4 & $143376\equiv 6\pmod{9}$ & $6\equiv 6\pmod{9}$ & $-12\equiv 6\pmod{9}$ \\
7 & $8288256\equiv 3\pmod{9}$ & $12\equiv 3\pmod{9}$ & $-24\equiv 3\pmod{9}$ \\
19 & $392037661056 \equiv 3\pmod{9}$ & $12\equiv 3\pmod{9}$ &  $-24\equiv 3\pmod{9}$ \\
31 & $779869748441088\equiv 0\pmod{9}$ & $36\equiv 0\pmod{9}$ & $-72\equiv 0\pmod{9}$ \\
\hline
\end{tabular}
\caption{$p=3$}
\label{tblclass3}
\end{table}

\begin{table}[h]
\begin{tabular}{|r|c|c|c|}
\hline
$D$ & $\dim W_D$ & $\tr(g_5|W_D)$ & $-24H(D)$ \\
\hline
\hline
3 & $26752\equiv 2\pmod{5}$ & $2\equiv 2\pmod{5}$ & $-8\equiv 2\pmod{5}$ \\
7 & $8288256\equiv 1\pmod{5}$ & $6\equiv 1\pmod{5}$ & $-24\equiv 1\pmod{5}$ \\
23 & $6103910176768\equiv 3\pmod{5}$ & $18\equiv 3\pmod{5}$ & $-72\equiv 3\pmod{5}$ \\
47 & $2548919136928993280\equiv 0\pmod{5}$ & $30\equiv 0\pmod{5}$ & $-120\equiv 0\pmod{5}$ \\
\hline
\end{tabular}
\caption{$p=5$}
\label{tblclass5}
\end{table}

\begin{table}[H]
\begin{tabular}{|r|c|c|c|}
\hline
$D$ & $\dim W_D$ & $\tr(g_7|W_D)$ & $-24H(D)$ \\
\hline
\hline
4 & $143376\equiv 2\pmod{7}$ & $2\equiv 2\pmod{7}$ & $-12\equiv 2\pmod{7}$ \\
8 & $26124256\equiv 4\pmod{7}$ & $4\equiv 4\pmod{7}$ & $-24\equiv 4\pmod{7}$ \\
11 & $561346944\equiv 4\pmod{7}$ & $4\equiv 4\pmod{7}$ & $-24\equiv 4\pmod{7}$ \\
15 & $18508941312\equiv 1\pmod{7}$ & $8\equiv 1\pmod{7}$ & $-48\equiv 1\pmod{7}$ \\
71 & $49186850301388438689792\equiv 0\pmod{7}$ & $28\equiv 0\pmod{7}$ & $-168\equiv 0\pmod{7}$\\
\hline
\end{tabular}
\caption{$p=7$}
\label{tblclass7}
\end{table}
\end{center}
\subsection{Selmer and Tate--Shafarevich Group Congruences}

Theorems~\ref{CongruencesSelmer} and \ref{USCorollary} offer criteria for detecting elements in $p$-Selmer groups and Tate--Shafarevich groups of quadratic twists of certain elliptic curves. \Cref{CongruencesSelmer}  assumes the truth of the
Birch and Swinnerton-Dyer Conjecture. Theorem~\ref{USCorollary} is unconditional thanks to results of Skinner--Urban.

Here we offer data related to the curves $E_{14}$ and $E_{15}$. In the notation of \Cref{USCorollary}, we consider fundamental discriminants $-D$ such that $\left(\frac{-D}{p}\right)=-1$ and $\left(\frac{-D}{p'}\right)=1$. For convenience let
\begin{align*}
H_{14}(D)&:=\delta_7(H(D)-\delta_7H^{(2)}(D)),\ \ \ \ \ \ 
&H_{15}(D)&:=\delta_5(H(D)-\delta_5H^{(3)}(D)),\\
\tr_2(D)&:=\tr(g_2|W_D),\ \ \ \ \ \ 
&\tr_3(D)&:=\tr(g_3|W_D),\\
\Diff_{14}(D)&:=H_{14}(D)-\tr_2(D),\ \ \ \ \ \
& \Diff_{15}(D)&:=H_{15}(D)-\tr_3(D).
\end{align*}
We have the following numerics.
In \Cref{tblE14,tblE15}, the second and third columns offer graded traces and differences of class numbers. The fourth and fifth
columns offer Mordell--Weil ranks and orders of Tate--Shafarevich groups assuming the Birch and Swinnerton-Dyer Conjecture.
By Theorem~\ref{USCorollary}, these columns are congruent if and only the corresponding $p$-Selmer group is nontrivial. First note that if these two columns are incongruent, then both the Mordell--Weil rank over $\Q$ and the $p$-part of the Tate--Shafarevich groups are trivial. However, when these
columns are congruent, notice that either the rank is positive or the Tate--Shafarevich group is nontrivial at $p$.

\begin{center}
\begin{table}[H]
\begin{tabular}{|r|r|c|c|c|c|}
\hline
$D$ & $\tr_2(D)$ & $H_{14}(D)$ & $\Diff_{14}(D)\smod{7}$ & $\rk(E_{14}(-D))$ & $\#\Sha_{an}(E_{14}(-D))$ \\
\hline
\hline
15 & -96256 & -30 & 3 & 0 & 1\\
23 & -1746944 & -45 & 0 & 2 & 1\\
39 & -165767168 & -60 & 4 & 0 & 1\\
71 & -156822906880 & -105 & 4 & 0 & 1\\
79 & -669595144192 & -75 & 3 & 0 & 1\\
239 & -6190369...040 & -225 & 0 & 2 & 1\\
2671 & -1630362...664 & -345 & 0 & 0 & 49 \\
\hline
\end{tabular}
\caption{Examples for the curve $E_{14}$}
\label{tblE14}
\end{table}
\end{center}

\begin{center}
\begin{table}[H]
\begin{tabular}{|r|r|c|c|c|c|}
\hline
$D$ & $\tr_3(D)$ & $H_{15}(D)$ & $\Diff_{15}(D)\smod{5}$ & $\rk(E_{15}(-D))$ & $\#\Sha_{an}(E_{15}(-D))$ \\
\hline
\hline
8 & -188 & -6 & 3 & 0 & 1\\
23 & -11456 & -18 & 2 & 0 & 1\\
47 & -860032 & -30 & 3 & 0 & 1\\
68 & -15834144 & -24 & 0 & 2 & 1\\
83 & -96763256 & -18 & 2 & 0 & 1\\
248 & -10546706...288 & -48 & 0 & 2 & 1\\
308 & -45931281...288 & -48 & 0 & 2 &  1\\
587 & -54506997...592 & -42 & 0 & 0 & 25 \\
1523 & -15706167...792 & -42 & 0 & 0 & 25 \\
\hline
\end{tabular}
\caption{Examples for the curve $E_{15}$}
\label{tblE15}
\end{table}
\end{center}

The authors thank Drew Sutherland for computing the elliptic curve invariants in \Cref{tblE14,tblE15}.

\appendix
\section{The Character Table of $\ON$}
Here we give the character table of the O'Nan group $\ON$ over the complex numbers. For $n\in\N$ we let $\zeta_n:=e^\frac{2\pi i}{n}$ and define
\begin{align*}
& A:=\frac{1+3\sqrt{5}}{2},\qquad 
B:=\sqrt{2},  \\
&C:=-\zeta_{19}-\zeta_{19}^7-\zeta_{19}^8-\zeta_{19}^{11}-\zeta_{19}^{12}-\zeta_{19}^{18},  \\
&D:=-\zeta_{19}^4-\zeta_{19}^6-\zeta_{19}^9-\zeta_{19}^{10}-\zeta_{19}^{13}-\zeta_{19}^{15},\\
&E:=-\zeta_{19}^2-\zeta_{19}^3-\zeta_{19}^5-\zeta_{19}^{14}-\zeta_{19}^{16}-\zeta_{19}^{17},\\
&F:=i\sqrt{5},\qquad 
G:=\sqrt{7},\qquad 
H:=\frac{-1+i\sqrt{31}}{2}.
\end{align*}
We use $\overline{A}$, $\overline{B}$, \&c. to denote images under the obvious Galois involutions. Note that $C,D,$ and $E$ are in one Galois orbit as well, since
\[(X-C)(X-D)(X-E)=X^3-X^2-6X+7.\]
The character table is reproduced from \textsc{Gap4} \cite{gap}.
\newpage
\begin{landscape}
\hspace{4cm}
\begin{table}[htp]
\tiny
\setlength{\tabcolsep}{2.5pt}
\begin{tabular}{|r|rrrrrrrrrrrrrrrrrrrrrrrrrrrrrr|}
\hline
 & $1A$ & $2A$ & $3A$ & $4A$ & $4B$ & $5A$ & $6A$ & $7A$ & $7B$ & $8A$ & $8B$ & $10A$ & $11A$ & $12A$ & $14A$ & $15A$ & $15B$ & $16A$ & $16B$ & $16C$ & $16D$ & $19A$ & $19B$ & $19C$ & $20A$ & $20B$ & $28A$ & $28B$ & $31A$ & $31B$  \\
\hline
$\chi_{1}$  &        1 &   1 &  1 &  1 & 1 & 1 & 1 & 1 & 1 & 1 & 1 &  1 &  1 &  1 &  1 &  1 &  1 &  1 & 1 & 1 & 1 &    1 &  1 &  1 &  1 &  1 &  1 &  1 &  1 &  1 \\
$\chi_{2}$  &    10944 &  64 &  9 & 64 & 0 &-1 & 1 &17 & 3 & 0 & 0 & -1 & -1 &  1 &  1 & -1 & -1 &  0 & 0 & 0 & 0 &    0 &  0 &  0 & -1 & -1 &  1 &  1 &  1 &  1\\ 
$\chi_{3}$  &    13376 & -64 & 11 & 64 & 0 & 1 &-1 &-1 &-1 & 0 & 0 &  1 &  0 &  1 & -1 &  1 &  1 &  0 & 0 & 0 & 0 &    0 &  0 &  0 & -1 & -1 &  1 &  1 &  $H$ & $\overline{H}$\\
$\chi_{4}$  &    13376 & -64 & 11 & 64 & 0 & 1 &-1 &-1 &-1 & 0 & 0 &  1 &  0 &  1 & -1 &  1 &  1 &  0 & 0 & 0 & 0 &    0 &  0 &  0 & -1 & -1 &  1 &  1 & $\overline{H}$ &  $H$\\
$\chi_{5}$  &    25916 & -36 & -4 & 20 & 4 & 1 & 0 &-5 & 2 & 0 & 0 & -1 &  0 &  2 & -1 &  1 &  1 &  0 & 0 & 0 & 0 &    0 &  0 &  0 &  $F$ & $-F$ & -1 & -1 &  0 &  0\\
$\chi_{6}$  &    25916 & -36 & -4 & 20 & 4 & 1 & 0 &-5 & 2 & 0 & 0 & -1 &  0 &  2 & -1 &  1 &  1 &  0 & 0 & 0 & 0  &    0 &  0 &  0 & $-F$ &  $F$ & -1 & -1 &  0 &  0\\
$\chi_{7}$  &    26752 & 128 & 22 &  0 & 0 & 2 & 2 &-2 &-2 & 0 & 0 & -2 &  0 &  0 &  2 &  2 &  2 &  0 & 0 & 0 & 0 &    0 &  0 &  0 &  0 &  0 &  0 &  0 & -1 & -1\\
$\chi_{8}$  &    32395 &  75 & -5 & 35 & 3 & 0 & 3 & 6 &-1 & 3 &-1 &  0 &  0 & -1 & -2 &  0 &  0 &  1 & 1 & -1 & -1  &    0 &  0 &  0 &  0 &  0 &  0 &  0 &  0 &  0 \\
$\chi_{9}$  &    32395 &  75 & -5 & 35 & 3 & 0 & 3 & 6 &-1 &-1 & 3 &  0 &  0 & -1 & -2 &  0 &  0 & -1 & -1 & 1 & 1  &    0 &  0 &  0 &  0 &  0 &  0 &  0 &  0 &  0 \\
$\chi_{10}$ &    37696 & -64 & 31 &-64 & 0 & 1 &-1 &15 & 1 & 0 & 0 &  1 & -1 & -1 & -1 &  1 &  1 &  0 & 0 & 0 & 0  &    0 &  0 &  0 &  1 &  1 & -1 & -1 &  0 &  0\\
$\chi_{11}$ &    52668 &  92 & 18 & 20 & 4 &-2 & 2 &-7 & 0 & 0 & 0 &  2 &  0 &  2 &  1 & -2 & -2 &  0 & 0 & 0 & 0 &    0 &  0 &  0 &  0 &  0 & -1 & -1 & -1 & -1\\
$\chi_{12}$ &    58311 &  71 & -9 & 71 & 7 & 1 &-1 & 1 & 1 &-1 &-1 &  1 &  0 &-1  &  1 &  1 &  1 & -1 & -1 & -1 & -1 &    0 &  0 &  0 &  1 &  1 &  1 &  1 &  0 &  0 \\
$\chi_{13}$ &    58311 &  71 & -9 & -1 &-1 & 1 &-1 & 1 & 1 & 3 &-1 &  1 &  0 & -1 &  1 &  1 &  1 & -1 & -1 & 1 & 1 &    0 &  0 &  0 & -1 & -1 & -1 & -1 &  0 &  0\\
$\chi_{14}$ &    58311 &  71 & -9 & -1 &-1 & 1 &-1 & 1 & 1 &-1 & 3 &  1 &  0 & -1 &  1 &  1 &  1 &  1 & 1 & -1 & -1 &    0 & 0  &  0 & -1 & -1 & -1 & -1 &  0 &  0\\
$\chi_{15}$ &    58653 & -35 &  9 &-35 &-3 & 8 & 1 & 0 & 0 & 1 & 1 &  0 &  1 &  1 &  0 & -1 & -1 & -1 & -1 & -1 & -1 &    0 & 0  &  0 &  0 &  0 &  0 &  0 &  1 &  1\\
$\chi_{16}$ &    64790 &  70 &-10 &-70 & 2 & 0 &-2 &12 &-2 & 0 & 0 &  0 &  0 &  2 &  0 &  0 &  0 &  $B$ & $-B$ & $-B$ & $B$ &    0 &  0 &  0 &  0 &  0 &  0 &  0 &  0 &  0 \\
$\chi_{17}$ &    64790 &  70 &-10 &-70 & 2 & 0 &-2 &12 &-2 & 0 & 0 &  0 &  0 &  2 &  0 &  0 &  0 & $-B$ & $B$ & $B$ & $-B$  &    0 &  0 &  0 &  0 &  0 &  0 &  0 &  0 &  0\\
$\chi_{18}$ &    85064 & -56 & 14 &-56 & 8 &-1 &-2 & 0 & 0 & 0 & 0 & -1 &  1 & -2 &  0 & -1 & -1 &  0 & 0 & 0 & 0 &    1 &  1 &  1 & -1 & -1 &  0 &  0 &  0 &  0\\
$\chi_{19}$ &   116963 &  35 & -1 & 35 & 3 & 8 &-1 & 0 & 0 &-1 &-1 &  0 &  0 & -1 &  0 & -1 & -1 &  1 & 1 & 1 & 1  &   -1 & -1 & -1 &  0 &  0 &  0 &  0 &  0 &  0\\
$\chi_{20}$ &   143374 &  14 &  4 &126 &-2 &-1 &-4 & 0 & 0 & 2 & 2 & -1 &  0 &  0 &  0 & -1 & -1 &  0 & 0 & 0 & 0 &    0 &  0 &  0 &  1 &  1 &  0 &  0 & -1 & -1\\
$\chi_{21}$ &   169290 &  90 &  0 &-90 &-2 & 0 & 0 &-5 & 2 & 0 & 0 &  0 &  0 &  0 & -1 &  0 &  0 &  $B$ & $-B$ & $B$ & $-B$ &    0 &  0 &  0 &  0 &  0 &  1 &  1 & -1 & -1\\
$\chi_{22}$ &   169290 &  90 &  0 &-90 &-2 & 0 & 0 & 5 & 2 & 0 & 0 &  0 &  0 &  0 & -1 &  0 &  0 & $-B$ & $B$ & $-B$ & $B$ &    0 &  0 &  0 &  0 &  0 &  1 &  1 & -1 & -1\\
$\chi_{23}$ &   175616 &   0 &  8 &  0 & 0 &-4 & 0 & 0 & 0 & 0 & 0 &  0 &  1 &  0 &  0 &  $A$ & $\overline{A}$ &  0 & 0 & 0 & 0 &   -1 & -1 & -1 &  0 &  0 &  0 &  0 & 1  & 1\\ 
$\chi_{24}$ &   175616 &   0 &  8 &  0 & 0 &-4 & 0 & 0 & 0 & 0 & 0 &  0 &  1 &  0 &  0 & $\overline{A}$ &  $A$ &  0 & 0 & 0 & 0 &   -1 & -1 & -1 &  0 &  0 &  0 &  0 &  1 &  1\\
$\chi_{25}$ &   175770 &  90 &  0 & 90 &-6 & 0 & 0 & 7 & 0 & 2 &-2 &  0 &  1 &  0 & -1 &  0 &  0 &  0 & 0 & 0 & 0 &    1 &  1 &  1 &  0 &  0 & -1 & -1 &  0 &  0 \\ 
$\chi_{26}$ &   207360 &   0 &  0 &  0 & 0 & 0 & 0 &-8 &-1 & 0 & 0 &  0 & -1 &  0 &  0 &  0 &  0 &  0 & 0 & 0 & 0&   $C$ &  $E$ &  $D$ &  0 &  0 &  0 &  0 &  1 &  1\\
$\chi_{27}$ &   207360 &   0 &  0 &  0 & 0 & 0 & 0 &-8 &-1 & 0 & 0 &  0 & -1 &  0 &  0 &  0 &  0 &  0 & 0 & 0 & 0 &   $D$ &  $C$ &  $E$ &  0 &  0 &  0 &  0 &  1 &  1\\
$\chi_{28}$ &   207360 &   0 &  0 &  0 & 0 & 0 & 0 &-8 &-1 & 0 & 0 &  0 & -1 &  0 &  0 &  0 &  0 &  0 & 0 & 0 & 0 &   $E$ &  $D$ &  $C$ &  0 &  0 &  0 &  0 &  1 &  1\\
$\chi_{29}$ &   234080 &-160 &-10 &  0 & 0 & 0 & 2 & 7 & 0 & 0 & 0 &  0 &  0 &  0 &  1 &  0 &  0 &  0 & 0 & 0 & 0 &    0 &  0 &  0 &  0 &  0 &  $G$ & $-G$ & -1 & -1 \\
$\chi_{30}$ &   234080 &-160 &-10 &  0 & 0 & 0 & 2 & 7 & 0 & 0 & 0 &  0 &  0 &  0 &  1 &  0 &  0 &  0 & 0 & 0 & 0 &    0 &  0 &  0 &  0 &  0 & $-G$ & $G$ & -1 & -1\\
\hline
\end{tabular}
\caption{Character table of $\ON$}
\label{ct1}
\end{table}
\end{landscape}

\newpage
\section{Multiplicities of Irreducible Representations in $W$}
We denote by $V_j$ the $\ON$-module corresponding to the irreducible $\chi_j$ in \Cref{ct1}. 

The following table gives the multiplicities of $V_j$ in the (virtual) modules $W_m$ in \Cref{thm:moon}. Negative multiplicities are printed in bold.

\begin{center}
\begin{table}[htp]
\tiny

\setlength{\tabcolsep}{2.5pt}
\begin{tabular}{|r||rrrrrrrrrr|}
\hline
$m$ & $V_1$ & $V_2$ & $V_3$ & $V_4$ & $V_5$ & $V_6$ & $V_7$ & $V_8$ & $V_9$ & $V_{10}$   \\
\hline
 3 & 0 & 0 & 0 & 0 & 0 & 0 & 1 & 0 & 0 & 0 \\
 4 & 1 & 0 & 0 & 0 & 0 & 0 & 0 & 0 & 0 & 0 \\
 7 & 0 & 0 & 1 & 1 & 1 & 1 & {\bf -2} & 0 & 0 & 2 \\
8 & {\bf -2} & 1 & 0 & 0 & 2 & 2 & 0 & 2 & 2 & 1 \\
11 & 0 & 18 & 8 & 8 & 28 & 28 & 40 & 48 & 48 & 34 \\
12 & {\bf -1} & 33 & 44 & 44 & 76 & 76 & 88 & 98 & 98 & 122 \\
15 & 0 & 406 & 581 & 581 & 1061 & 1061 & 1010 & 1252 & 1252 & 1568 \\
16 & {\bf -2} & 978 & 1193 & 1193 & 2316 & 2316 & 2386 & 2892 & 2892 & 3362 \\
19 & 2 & 9484 & 11205 & 11205 & 21948 & 21948 & 23114 & 27766 & 27766 & 31894 \\
20 & 5 & 18951 & 23161 & 23161 & 44930 & 44930 & 46322 & 56156 & 56156 & 65271 \\
23 & 2 & 144238 & 177831 & 177831 & 343685 & 343685 & 352892 & 428308 & 428308 & 499900 \\
24 & 25 & 277191 & 338794 & 338794 & 656282 & 656282 & 677588 & 820362 & 820362 & 954783 \\
27 & 212 & 1795740 & 2189365 & 2189365 & 4245047 & 4245047 & 4388491 & 5310882 & 5310882 & 6174470 \\
28 & 292 & 3264537 & 3989983 & 3989983 & 7730566 & 7730566 & 7979966 & 9663217 & 9663217 & 11244510 \\
31 & 1562 & 18513448 & 22644956 & 22644956 & 43863830 & 43863830 & 45258570 & 54815104 & 54815104 & 63803360 \\
32 & 2960 & 32416998 & 39620773 & 39620773 & 76765848 & 76765848 & 79241546 & 95957290 & 95957290 & 111658534 \\
35 & 15432 & 165271652 & 201946677 & 201946677 & 391304807 & 391304807 & 403986962 & 489174874 & 489174874 & 569165006 \\
36 & 25645 & 279985728 & 342204752 & 342204752 & 663020690 & 663020690 & 684409504 & 828775828 & 828775828 & 964395212 \\
      \hline      
\end{tabular}
\caption{Multiplicities, part I.}
\label{mults1}
\end{table}
\end{center}

\begin{center}
\begin{table}[htp]
\tiny

\setlength{\tabcolsep}{2.5pt}
\begin{tabular}{|r||rrrrrrrrrr|}
\hline
$m$ & $V_{11}$ & $V_{12}$ & $V_{13}$ & $V_{14}$ & $V_{15}$ & $V_{16}$ & $V_{17}$ & $V_{18}$ & $V_{19}$ & $V_{20}$   \\
\hline
3 & 0 & 0 & 0 & 0 & 0 & 0 & 0 & 0 & 0 & 0 \\
4 & 0 & 1 & 0 & 0 & 0 & 0 & 0 & 1 & 0 & 0 \\
7 & 0 & 0 & 0 & 0 & 2 & 0 & 0 & 2 & 2 & 2 \\
8 & 2 & 2 & 4 & 4 & 3 & 4 & 4 & 2 & 7 & 8 \\
11 & 72 & 80 & 80 & 80 & 64 & 88 & 88 & 96 & 144 & 176 \\
12 & 164 & 173 & 178 & 178 & 185 & 197 & 197 & 261 & 359 & 444 \\
15 & 2068 & 2296 & 2296 & 2296 & 2384 & 2556 & 2556 & 3458 & 4680 & 5754 \\
16 & 4704 & 5210 & 5200 & 5200 & 5224 & 5782 & 5782 & 7598 & 10432 & 12788 \\
19 & 45058 & 49802 & 49802 & 49802 & 49804 & 55314 & 55314 & 72214 & 99604 & 122014 \\
20 & 91248 & 101087 & 101068 & 101068 & 101628 & 112302 & 112302 & 147407 & 202710 & 248454 \\
23 & 696576 & 771644 & 771644 & 771644 & 777260 & 857476 & 857476 & 1127304 & 1548902 & 1898946 \\
24 & 1333868 & 1476646 & 1476680 & 1476680 & 1485435 & 1640744 & 1640744 & 2154259 & 2962056 & 3630946 \\
27 & 8633536 & 9557140 & 9557140 & 9557140 & 9609292 & 10618702 & 10618702 & 13936084 & 19166220 & 23493012 \\
28 & 15710534 & 17393783 & 17393848 & 17393848 & 17495880 & 19326474 & 19326474 & 25374046 & 34889380 & 42767664 \\
31 & 89122420 & 98675012 & 98675012 & 98675012 & 99266748 & 109640000 & 109640000 & 143966514 & 197940198 & 242639964 \\
32 & 156007392 & 172723134 & 172723024 & 172723024 & 173735642 & 191914504 & 191914504 & 251967626 & 346455808 & 424687592 \\
35 & 795291752 & 880491400 & 880491400 & 880491400 & 885616006 & 978320518 & 978320518 & 1284400160 & 1766091974 & 2164876128 \\
36 & 1347430236 & 1491796517 & 1491796318 & 1491796318 & 1500546542 & 1657551544 & 1657551544 & 2176231689 & 2992317414 & 3668002182 \\
      \hline      
\end{tabular}
\caption{Multiplicities, part II.}
\label{mults2}
\end{table}
\end{center}

\begin{center}
\begin{table}[htp]
\tiny

\setlength{\tabcolsep}{2.5pt}
\begin{tabular}{|r||rrrrrrrrrr|}
\hline
$m$ & $V_{21}$ & $V_{22}$ & $V_{23}$ & $V_{24}$ & $V_{25}$ & $V_{26}$ & $V_{27}$ & $V_{28}$ & $V_{29}$ & $V_{30}$   \\
\hline
3 & 0 & 0 & 0 & 0 & 0 & 0 & 0 & 0 & 0 & 0 \\
4 & 0 & 0 & 0 & 0 & 0 & 0 & 0 & 0 & 0 & 0 \\
7 & 2 & 2 & 3 & 3 & 2 & 4 & 4 & 4 & 6 & 6 \\
8 & 10 & 10 & 9 & 9 & 10 & 12 & 12 & 12 & 14 & 14 \\
11 & 216 & 216 & 214 & 214 & 224 & 252 & 252 & 252 & 270 & 270 \\
12 & 521 & 521 & 543 & 543 & 542 & 638 & 638 & 638 & 718 & 718 \\
15 & 6746 & 6746 & 7057 & 7057 & 7006 & 8328 & 8328 & 8328 & 9492 & 9492 \\
16 & 15102 & 15102 & 15671 & 15671 & 15680 & 18500 & 18500 & 18500 & 20884 & 20884 \\
19 & 144268 & 144268 & 149402 & 149402 & 149782 & 176414 & 176414 & 176414 & 198703 & 198703 \\
20 & 293374 & 293374 & 304323 & 304323 & 304596 & 359352 & 359352 & 359352 & 405676 & 405676 \\
23 & 2241422 & 2241422 & 2326161 & 2326161 & 2327256 & 2746666 & 2746666 & 2746666 & 3102368 & 3102368 \\
24 & 4287248 & 4287248 & 4447476 & 4447476 & 4451367 & 5251357 & 5251357 & 5251357 & 5927992 & 5927992 \\
27 & 27742332 & 27742332 & 28775511 & 28775511 & 28804106 & 33976834 & 33976834 & 33976834 & 38348849 & 38348849 \\
28 & 50498270 & 50498270 & 52385258 & 52385258 & 52431245 & 61854317 & 61854317 & 61854317 & 69824744 & 69824744 \\
31 & 286490080 & 286490080 & 297207048 & 297207048 & 297456630 & 350929578 & 350929578 & 350929578 & 396169260 & 396169260 \\
32 & 501453364 & 501453364 & 520191449 & 520191449 & 520647692 & 614220424 & 614220424 & 614220424 & 693367868 & 693367868 \\
35 & 2556221884 & 2556221884 & 2651707883 & 2651707883 & 2654066434 & 3131025718 & 3131025718 & 3131025718 & 3534425359 & 3534425359 \\
36 & 4331022760 & 4331022760 & 4492864127 & 4492864127 & 4496803456 & 5304984880 & 5304984880 & 5304984880 & 5988574304 & 5988574304 \\
      \hline      
\end{tabular}
\caption{Multiplicities, part III.}
\label{mults3}
\end{table}
\end{center}
\newpage

\section{Congruences}\label{AppCong}

\underline{$p=31$:}
\begin{align*}
0&\equiv F_{1A}-F_{31AB} &\pmod{31}
\end{align*}

\underline{$p=19$:}
\begin{align*}
0&\equiv F_{1A}-F_{19ABC} &\pmod{19}
\end{align*}

\underline{$p=11$:}
\begin{align*}
0&\equiv F_{1A}-F_{11A} &\pmod{11}
\end{align*}

\underline{$p=7$:}
\begin{align*}
0&\equiv F_{1A}-F_{7AB} &\pmod{7^3}\\
 &\equiv F_{2A}-F_{14A} &\pmod{7}\\
 &\equiv F_{4AB}-F_{28AB} &\pmod{7}
\end{align*}

\underline{$p=5$:}
\begin{align*}
0&\equiv F_{1A}-F_{5A} &\pmod{5^3}\\
 &\equiv F_{2A}-F_{10A} &\pmod{5}\\
 &\equiv F_{3A}-F_{15AB} &\pmod{5}\\
 &\equiv F_{4AB}-F_{20AB} &\pmod{5}\\
\end{align*}

\underline{$p=3$:}
\begin{align*}
0&\equiv F_{1A}-F_{3A} &\pmod{3^5}\\
 &\equiv F_{2A}-F_{6A} &\pmod{3^2}\\
 &\equiv F_{4AB}-F_{12A} &\pmod{3^2}\\
 &\equiv F_{5A}-F_{15AB} &\pmod{3^2}
\end{align*}

\underline{$p=2$:}
\begin{align*}
0&\equiv F_{1A}+303F_{2A}+3024F_{4AB}+4864F_{8AB}+57344F_{16ABCD} &\pmod{2^{16}}\\
 &\equiv F_{2A}+7F_{4AB}+8F_{8AB}+112F_{16ABCD} &\pmod{2^{7}}\\
 &\equiv F_{3A}+F_{6A}+6F_{12A} &\pmod{2^{3}}\\
 &\equiv F_{4AB}+F_{8AB}+14F_{16ABCD} &\pmod{2^{4}}\\
 &\equiv F_{5A}+F_{10A}+6F_{20AB} &\pmod{2^{3}}\\
 &\equiv F_{6A}+F_{12A} &\pmod{2}\\
 &\equiv F_{7AB}+F_{14AB} &\pmod{2^{3}}\\
 &\equiv F_{8AB}+7F_{16ABCD} &\pmod{2^{3}}\\
 &\equiv F_{10A}+F_{20AB} &\pmod{2}\\
 &\equiv F_{14A}+F_{28AB} &\pmod{2}\\
 \end{align*}

\newpage 
\section{Traces of Singular Moduli}\label{appSing}
We give the explicit descriptions of $F_{[g]}$ in terms of traces of singular moduli and class numbers as described in \Cref{secSing}.

\begin{align*}
F_{1A}&=\calT^{(1)},\\
F_{2A}&=\calT^{(2)}+12\calH^{(1)}-12\calH^{(2)},\\
F_{3A}&=\calT^{(3)}+12\calH^{(1)}-12\calH^{(3)},\\
F_{4AB}&=\calT^{(4)}+12\calH^{(2)}-12\calH^{(4)},\\
F_{5A}&=\calT^{(5)}+6\calH^{(1)}-6\calH^{(5)},\\
F_{6A}&=\calT^{(6)}-12\calH^{(1)}+8\calH^{(2)}+\frac{21}2\calH^{(3)}-\frac{13}2\calH^{(6)},\\
F_{7AB}&=\calT^{(7)}+4\calH^{(1)}-4\calH^{(7)},\\
F_{8AB}&=\calT^{(8)}+4\calH^{(4)}-4\calH^{(8)},\\
F_{10A}&=\calT^{(10)}-6\calH^{(1)}+4\calH^{(2)}+\frac{11}2\calH^{(5)}-\frac{7}2\calH^{(10)},\\
F_{11A}&=\calT^{(11,+)}+\frac{12}5\calH^{(1)}-\frac{6}5\calH^{(11)}-\frac 45 \g^{(11)},\\
F_{12A}&=\calT^{(12)}-4\calH^{(2)}+4\calH^{(4)}+\frac 52\calH^{(6)}-\frac 52\calH^{(12)},\\
F_{14A}&=\calT^{(14,+)}-4\calH^{(1)}+\frac 83 \calH^{(2)}+\frac{15}{4}\calH^{(7)}-\frac{41}{24}\calH^{(14)}+\frac{8}{3} \g^{(14)},\\
F_{15AB}&=\calT^{(15,+)}-3\calH^{(1)}+\frac 94\calH^{(3)}+\frac 52\calH^{(5)}-\frac{13}{8}\calH^{(15)}+\frac 94 \g^{(15)},\\
F_{16ABCD}&=2\calT^{(32,+)}-\calT^{(16)}-2\calH^{(8)}+4\calH^{(16)}-\calH^{(32)},\\
F_{19ABC}&=\calT^{(19,+)}+\frac 43\calH^{(1)}-\frac 23\calH^{(19)}+\frac 43\g^{(19)},\\
F_{20AB}&=\calT^{(20,+)}-2\calH^{(2)}+2\calH^{(4)}+\frac 32\calH^{(10)}-\frac 32\calH^{(20)},\\
F_{28AB}&=\calT^{(28,+)}-\frac{4}{3}\calH^{(2)}+\frac{4}{3}\calH^{(4)}+\frac{25}{24}\calH^{(14)}-\frac{25}{24}\calH^{(28)}+\frac 83\g^{(28)},\\
F_{31AB}&=\calT^{(31,+)}+\frac 45 \calH^{(1)}-\frac 25\calH^{(31)}+\frac 35 \g^{(31)}.
\end{align*}

Here, $\g^{(N)}$ denotes the unique weight $3/2$ cusp form for  $\Gamma_0(4N)$ in the plus space with leading coefficient $1$ if $N<28$, $\g^{(28)}$ is the unique normalized cusp form in $S_{\frac 32}(\Gamma_0(28))$ hit with the $V_4$-operator, and $\g^{(31)}\in S^+_\frac 32(\Gamma_0(124))$ is the unique cusp form in this space satisfying
\[\g^{(31)}(\tau)=q^4+\frac{11}3q^7+O(q^8).\]
There is exactly one newform $f^{(31)}$ in $S^+_\frac 32(\Gamma_0(124))$ up to Galois conjugation and we have
\[f^{(31)}(\tau)=q^4-\frac{1+\sqrt{5}}{2}q^7-\frac{1-\sqrt{5}}{2}q^8+\frac{1+\sqrt{5}}{2}q^{16}+O(q^{20}),\]
so that we can express $g^{(31)}$ in terms of this newform as follows,
\[\g^{(31)}(\tau)=\frac{3+5\sqrt{5}}{6}f^{(31)}(\tau)+\frac{3-5\sqrt{5}}{6}{f^{(31)}}^\sigma(\tau),\]
where, as in \Cref{tbl:HMplus}, a superscript $\sigma$ denotes Galois conjugation.

\end{document}